\documentclass[a4paper]{article}
\usepackage{amsmath,amsthm,amssymb}
\usepackage{color}
\usepackage{tikz}
\usepackage{pgfplots,pgfplotstable}
\usepackage{tabularx}
\usepackage{setspace}
\usepackage{mathtools}
\usepackage{comment}
\usepackage[normalem]{ulem}
\usepackage{yhmath}
\usepackage{geometry} 
\usepackage{xspace}
\usepackage{dsfont}
  \def\ZR{{\mathbb R}}
  \def\ZQ{{\mathbb Q}}

  \def\ZN{{\mathbb N}}

\def\Ker#1{\mathop{\rm Ker}\nolimits#1}
\def\Imm#1{\mathop{\rm Im}\nolimits#1}

\def\div{\mathop{\rm div}\nolimits}

\def\norm#1.#2.{\|#1\|_{#2}}
\def\Norm#1.#2.{\big\|#1\big\|_{#2}}
\def\NOrm#1.#2.{\bigg\|#1\bigg\|_{#2}}
\def\NORm#1.#2.{\Big\|#1\Big\|_{#2}}
\def\NORM#1.#2.{\Bigg\|#1\Bigg\|_{#2}}

\def\vec#1{{\mathchoice{\mbox{\boldmath$\displaystyle#1$}}
{\mbox{\boldmath$\textstyle#1$}}
{\mbox{\boldmath$\scriptstyle#1$}}
{\mbox{\boldmath$\scriptscriptstyle#1$}}}}

\def \0b{{\hbox{\boldmath $0$}}}

 \newcommand{\bb}{\vec{b}}\def \b{\vec{b}}

\newcommand{\eb}{\vec{e}} \newcommand{\fb}{\vec{f}}

 \newcommand{\nb}{\vec{n}}
 
\newcommand{\qb}{\vec{q}} 
 \newcommand{\tb}{\vec{t}}
 \newcommand{\vb}{\vec{v}}
 \newcommand{\xb}{\vec{x}}
\newcommand{\yb}{\vec{y}} 

\newcommand{\Abb}{{\bf A}}

\newcommand{\Ibb}{{\bf I}} 
\newcommand{\Kbb}{{\bf K}} \newcommand{\Lbb}{{\bf L}}
\newcommand{\Mbb}{{\bf M}}

\newcommand{\Cb}{\vec{C}}

\newcommand{\Qb}{\vec{Q}}

\def \gamb{\vec{\gamma}} 
 
 \def \etab{\vec{\eta}}

 \def \xib{\vec{\xi}}

 \def \phib{\vec{\phi}}
 
\def \psib{\vec{\psi}}

\newcommand{\cC}{{\cal C}} \newcommand{\cD}{{\cal D}}
\newcommand{\cE}{{\cal E}} 
\newcommand{\cG}{{\cal G}} \newcommand{\cH}{{\cal H}}
\newcommand{\cI}{{\cal I}} 
 \newcommand{\cL}{{\cal L}}
 
 \newcommand{\cP}{{\cal P}}
 
 \newcommand{\cT}{{\cal T}}
 \newcommand{\cV}{{\cal V}}

\newcounter{primjer}[section]
\newcounter{tvrdnja}[section]
\newcounter{zadatak}[section]
\renewcommand\sharp{\#}

\newcommand\ie{\emph{i.e.}\xspace}
\newcommand\eg{\emph{e.g.}\xspace}

\newcommand\Ltwo{\mathrm{L}^2}
\newcommand\Linf{\mathrm{L}^\infty}

\newcommand\Hsp[1]{\mathrm{H}^{#1}}

\newcommand\Hone{\mathrm{H}^1}
\newcommand\Honezero{\mathrm{H}^1_0}
\newcommand\Htwo{\mathrm{H}^2}
\newcommand\xnorm[2]{\left\| #1 \right\|_{#2}}
\newcommand\bnorm[2]{\big\| #1 \big\|_{#2}}

\newcommand{\homg}{{\rm hom}}

\newcommand{\init}{{\rm init}}
\newcommand{\Omegad}{{\Gamma^\delta}}
\newcommand{\OmegadD}{{\Gamma^\delta_{\textrm{D}}}}
\newcommand{\Gammad}{{\Gamma_Y^\delta}}

\newcommand{\Hper}{\Hone_{\sharp}}
\newcommand{\Hperaver}{\Hone_{\star}(\Gamma_Y)}
\newcommand{\Loper}{\Ltwo(\Omega; \Ltwo(\Gamma_Y))}
\newcommand{\Hoper}{\Ltwo(\Omega; \Hper)}

\newcommand{\zerob}{\mathbf{0}}
\newcommand{\mud}{{\mu_\delta}}

\definecolor{mydarkcyan}{rgb}{0,0.5,0.5}
\definecolor{mygreen}{rgb}{0,0.7,0}
\definecolor{matkopurple}{RGB}{204,0,204}
\newcommand{\jtt}[1]{{\color{blue}#1}}

\newcommand{\jttd}[1]{\sout{{\bf\color{blue}#1}}}

\newcommand{\ks}[1]{{\color{mygreen}#1}}

\newcommand{\ksd}[1]{\sout{{\bf\color{mygreen}#1}}}

\newcommand{\ads}[1]{{\color{mydarkcyan}#1}}

\newcommand{\Abbh}{\Abb_\homg}
\newcommand{\ww}{w}
\newcommand{\tf}{t_{\mathsf{f}}}

\newcommand{{\cycle}}{{circuit}}
\newcommand{\cycles}{{circuits}}
\newcommand{{\trail}}{{trail}}
\newcommand{\trails}{{trails}}
\newcommand{\GYper}{\Gamma_{Y,\#}}
\let\leq=\leqslant

\newtheorem{theorem}{Theorem}[section]
\newtheorem{lemma}[theorem]{Lemma}
\newtheorem{corollary}[theorem]{Corollary}
\newtheorem{example}{Example}

\newtheorem{definition}[theorem]{Definition}

\newtheorem{remark}[theorem]{Remark}

\numberwithin{equation}{section}

\newcommand{\diag}{\mathop{\mathrm{diag}}}

\title{Homogenization of the time-dependent heat equation on planar one-dimensional periodic structures}
\author{Matko Ljulj\footnote{Department of Mathematics, Faculty of Science, University of Zagreb, Croatia,  \texttt{mljulj@math.hr}}, Kersten Schmidt\footnote{Fachbereich Mathematik, AG Numerik und Wissenschaftliches Rechnen, Technische Universit\"{a}t Darmstadt, Darmstadt, Germany}, Adrien Semin\addtocounter{footnote}{-1}\footnotemark{}\ \ and Josip Tamba\v{c}a\footnote{Department of Mathematics, Faculty of Science, University of Zagreb, Croatia, \texttt{tambaca@math.hr}}}
\date{\today}

\begin{document}

\maketitle

\begin{abstract}
  In this paper we consider the homogenization of a time-dependent heat conduction
  problem on a planar one-dimensional periodic structure.
  On the edges of a graph the one-dimensional heat equation is posed,
  while the Kirchhoff junction condition is applied at all (inner) vertices.
  Using the two-scale convergence adapted to homogenization of
  lower-dimensional problems we obtain the limit homogenized problem
  defined on a two-dimensional domain that is occupied by the mesh when the mesh period $\delta$ tends to $0$.
  The homogenized model is given by the classical heat equation
  with the  conductivity tensor depending on the unit cell graph only through
  the topology of the graph and lengthes of its edges.
  We show the well-posedness of the limit problem and give a purely algebraic formula for
  the computation of the homogenized conductivity tensor.
  The analysis is completed by numerical experiments showing
  a convergence to the limit problem where the convergence order in~$\delta$
  depends on the unit cell pattern. 
\end{abstract}

\tableofcontents

\section{Introduction}
\label{sec1}
In this paper we deal with homogenization of
the time-dependent heat conduction on a periodic graph
$\Omegad$ of period $\delta$ in a rectangular subdomain of $\ZR^2$
(see Figure~\ref{fig:example_Omega_delta} for examples of such domains~$\Omegad$
and Figure~\ref{fig:example_pattern_Gamma_Y}
for examples of unit cell pattern).
On edges of the graph the heat conduction is modeled by the
one-dimensional heat equation
\begin{align}
  \rho c_p \partial_t u^\delta + a \partial_{\Gamma}^2 u^\delta = f^\delta \quad \text{ on } \Omegad
  \label{eq:Heat_equation}
\end{align}
where $\partial_{\Gamma}^2$ is the second order derivative along the edges of $\Omegad$.
The Kirchhoff junction law \cite{Kuchment:2002,Rubinstein.Schatzman:2001} is applied at the (inner) vertices
and the problem is completed by initial and boundary condition
and will be detailed in Sec.~\ref{sec:description}.
Such a model on a lower-dimensional domain can be justified as limit model
of the heat conduction on a graph-like domain with thin edges -- also called fat graph
--
with the technique of~\cite{Kuchment:2002} or~\cite{Edo} (see also~\cite{Joly.Semin:2008} for a higher order model
for the wave equation).

For elliptic problems on lower-dimensional domains Bouchitt\'{e}~\emph{et.al.}~\cite{BBS, bouchitte2001homogenization,bouchitte2002homogenization, bouchitte2004homogenization}
and Zhikov~\cite{Zhikov} introduced and used an extension of the two-scale convergence
on measures to obtain limit problems. %
It is based on an extension of the two-scale convergence on surfaces
that was introduced by Neuss-Radu in~\cite{NeussRadu}
and in \cite{RaduDiploma} and by Allaire, Damlamian and Hornung in
\cite{AllaireDamlamianHornung} to incorporate boundary conditions with oscillating coefficients.
The two-scale convergence on measures was applied in
\cite{bouchitte2001homogenization}, \cite{bouchitte2002homogenization}
in a variational approach
using the $\Gamma$--convergence for the energy
functional and scalar problems.
In~\cite{bouchitte2004homogenization}
the technique of fattening of lower-dimensional problems is also applied
for the vector unknown variational problem. In~\cite{Zhikov2} the
measure generalization is used to analyze the elasticity problem on
singular structures.

The homogenization of the time-dependent heat equation on planar
one-dimensional periodic structures has not been studied
so far and shall be addressed in this work.
Using the \emph{mesh two-scale convergence} we rigorously derive the limit
homogenized model which is given by the heat conduction problem on the
two dimensional domain occupied by the mesh.
We focus on the homogenization of the domain so in order to simplify
the procedure we assume that the conductivity is scalar and
non-oscillatory. Furthermore we restrict to non-oscillatory initial
temperature.
As in the homogenization of the heat equation the time can be kept as a parameter in the problem, see \eg~\cite{CD}.
Thus it is not surprising that the homogenized conductivity tensor is the same as for
the stationary diffusion equation; compare with \cite{bouchitte2001homogenization},
see also~\cite{Panasenko} for a formal asymptotic expansion.
In difference to the previous works we will introduce a purely algebraic formula
for the computation of the homogenized conductivity tensor.
The homogenized tensor depends on the unit cell graph only through the topology of the unit cell graph
(incidence matrix of the graph) and the lengthes of the edges of the graph. Depending on the unit cell geometries
it can be a multiple of the unit matrix or a diagonal or non-diagonal matrix valued function.
Moreover, the homogenized conductivity tensor
turns out to be positive definite which then leads to the existence
and uniqueness result for the limit problem.

In \cite{CSJP} the homogenization on a $3D$ fattened graph for a simple
geometry like in Figure~\ref{fig:example_pattern_Gamma_Y}(a) is done first and
the thickness is taken to zero last.  In our approach we start with the
one-dimensional model (i.e. the thickness is first taken to zero) and then we
do the homogenization.  As already noted in~\cite{bouchitte2001homogenization}
the limit model for the stationary diffusion problem on fattened graphs in 3D
that is obtained when thickness of the fat edges and the period tend both to
zero is the same independently of the way the limiting is performed.  Thus it
is not a surprise that the models in \cite{CSJP} coincide with the two-scale
limit of the one-dimensional limit model.  However, it is important to consider
homogenization on the mesh objects directly to develop techniques where
thickening is not possible, either because three-dimensional equations are too
complicated or do not exist or existing techniques using an approach based on measures can not be applied.

\begin{figure}[!bt]
  \null\hfill
  \begin{tikzpicture}
    \draw[dashed, black!50!white] (0,0) rectangle (4,3);
    \pgfmathparse{0.5}\let\valuedelta\pgfmathresult;
    \pgfmathparse{4 / \valuedelta - 1}\let\nmaxx\pgfmathresult;
    \pgfmathparse{3 / \valuedelta - 1}\let\nmaxy\pgfmathresult;

    \foreach \posx in {0,...,\nmaxx}
    {
      \pgfmathparse{\posx * \valuedelta}\let\valuex\pgfmathresult;
      \begin{scope}[xshift = \valuex cm]
        \foreach \posy in {0,...,\nmaxy}
        {
          \pgfmathparse{\posy * \valuedelta}\let\valuey\pgfmathresult;
          \begin{scope}[yshift = \valuey cm]
            \draw (0,\valuedelta) ++ (-90:0.5*\valuedelta) arc(-90:-60:\valuedelta);
            \draw (\valuedelta,0) ++ (90:0.5*\valuedelta) arc(90:120:\valuedelta);
            \draw (0.5*\valuedelta,0) -- (0.5*\valuedelta,\valuedelta);
          \end{scope}
        }
      \end{scope}
    }
    \node at (2,-0.5) {(a) $\delta=0.5$};
  \end{tikzpicture}
  \hfill
  \begin{tikzpicture}
    \draw[dashed, black!50!white] (0,0) rectangle (4,3);
    \pgfmathparse{0.25}\let\valuedelta\pgfmathresult;
    \pgfmathparse{4 / \valuedelta - 1}\let\nmaxx\pgfmathresult;
    \pgfmathparse{3 / \valuedelta - 1}\let\nmaxy\pgfmathresult;

    \foreach \posx in {0,...,\nmaxx}
    {
      \pgfmathparse{\posx * \valuedelta}\let\valuex\pgfmathresult;
      \begin{scope}[xshift = \valuex cm]
        \foreach \posy in {0,...,\nmaxy}
        {
          \pgfmathparse{\posy * \valuedelta}\let\valuey\pgfmathresult;
          \begin{scope}[yshift = \valuey cm]
            \draw (0,\valuedelta) ++ (-90:0.5*\valuedelta) arc(-90:-60:\valuedelta);
            \draw (\valuedelta,0) ++ (90:0.5*\valuedelta) arc(90:120:\valuedelta);
            \draw (0.5*\valuedelta,0) -- (0.5*\valuedelta,\valuedelta);
          \end{scope}
        }
      \end{scope}
    }
    \node at (2,-0.5) {(b) $\delta=0.25$};
  \end{tikzpicture}
  \hfill\null
  \caption{Example of one-dimensional periodic domain $\Omegad$
    for values of $\delta$ in the planar rectangle $\Omega$ with side lengthes  $L_1 = 4$, $L_2 = 3$
    for the pattern configuration  shown in Fig.~\ref{fig:example_pattern_Gamma_Y}(b).}
  \label{fig:example_Omega_delta}
\end{figure}
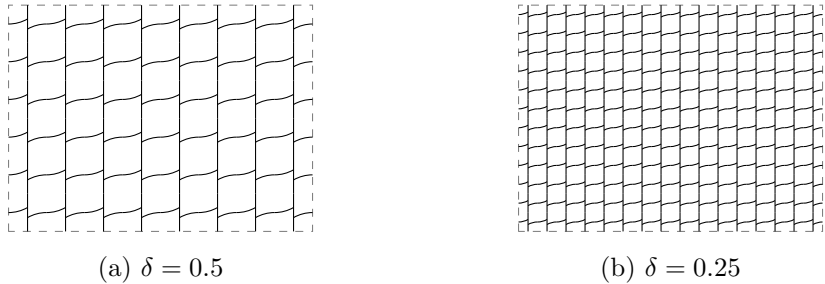

The article is organized as follows. In Section~\ref{sec:description} we
describe the problem in detail and present the main results: %
a priori stability estimates, the homogenized model and weak convergence as well as convergence in norm. In order to do asymptotics
of the problem we need the compactness results. We use the compactness result
for $\Ltwo(\Omegad)$ from \cite{RaduDiploma} and derive the compactness result
for $\Hone(\Omegad)$ and for $\partial_t u^\delta$ in $\Ltwo(\Omegad)$. Some
results from the graph theory give us technical conditions that allow us to
prove the compactness theorem. Then in Section~\ref{sec:homogenized_model:convergence_in_norm} we prove the
\textit{a priori} estimates for $u^\delta$, the solution of the problem on the
one-dimensional mesh~$\Omegad$.
In Section~\ref{sec:tp} we prove some technical results and in Section~\ref{sec:homogenized_model:properties} the main properties of the homogenized model.
Using two-scale convergences that follow from
the compactness theorems we are able to make two limits in the $\delta$-problem
and obtain the equation for the corrector and the limit problem. The corrector
equation leads to the formulation of the canonical problem which is then used
to build the homogenized tensor (Section~\ref{sec:homogenized_model:convergence}).
In Section~\ref{sec:homogenized_model:convergence_in_norm} we prove the corresponding convergence in norm.
In Section~\ref{sec5} we formulate the
algebraic method to compute the homogenized tensor and compute it for five
different patterns. In Section~\ref{sec:numerics}
we compare by numerical experiments the solution of the heat equation on the $\delta$-periodic graph
with the one of the homogenized problem.

\section{Description of the problem and the main result}
\label{sec:description}

In this section we formulate the model including the geometrical setting we consider in this
paper. We start with the description of the geometry.

\begin{figure}[bt]
  \vspace*{-1cm}
  \null\hfill
  \begin{tikzpicture}
    \draw[dashed, black!50!white] (-2,-2) rectangle (2,2);
    \draw[white] (-2,4) -- (2,-4);
    \draw (-2,0) -- (2,0);
    \draw (0,-2) -- (0,2);

    \draw[thick,->] (-1.5,0) -- (-0.5,0) node [pos=0.5,below] {$\tb$};
    \draw[thick,->] (0.5,0) -- (1.5,0) node [pos=0.5,below] {$\tb$};
    \draw[thick,->] (0,-1.5) -- (0,-0.5) node [pos=0.5,right] {$\tb$};
    \draw[thick,->] (0,0.5) -- (0,1.5) node [pos=0.5,right] {$\tb$};

    \node at (0,-2.5) {(a)};
  \end{tikzpicture}
  \hfill
  \begin{tikzpicture}
    \draw[dashed, black!50!white] (-2,-2) rectangle (2,2);
    \draw (-2,4) ++ (-90:4) arc(-90:-60:4);
    \draw (2,-4) ++ (90:4) arc(90:120:4);
    \draw (0,-2) -- (0,2);

    \draw[thick,->] (-2,4) ++ (-80:4) arc(-80:-70:4) node [pos=0.5,below] {$\tb$};
    \draw[thick,->] (2,-4) ++ (110:4) arc(110:100:4) node [pos=0.5,below] {$\tb$};
    \draw[thick,->] (0,-1.5) -- (0,-0.9) node [pos=0.5,right] {$\tb$};
    \draw[thick,->] (0,-0.3) -- (0,0.3) node [pos=0.5,right] {$\tb$};
    \draw[thick,->] (0,0.9) -- (0,1.5) node [pos=0.5,right] {$\tb$};

    \node at (0,-2.5) {(b)};
  \end{tikzpicture}
  \hfill\null
  \vspace*{-1cm}
  \caption{Two different examples of valid patterns $\Gamma_Y$}
  \label{fig:example_pattern_Gamma_Y}
\end{figure}
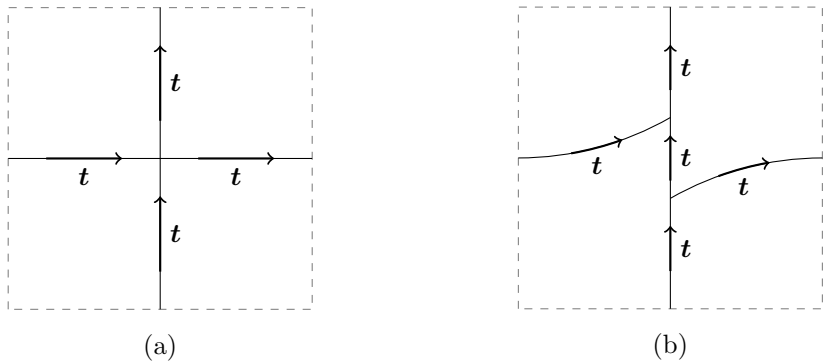

Let us consider a one-dimensional manifold $\Gamma_Y \subset [0,1]^2$
--~the \emph{unit mesh}~-- that can be described by a connected and oriented
graph structure $(\cV,\cE)$ (please see~\cite{MR0411988,MR2524249} for the
usual terminology in the graph theory), where $\cV$ and $\cE$ denote
respectively the set of vertices and the set of edges. We assume that each edge
$e \in \cE$ admits a $W^{1,\infty}$ and bijective arc length parametrization
$\gamb$ from an open interval, that is directed in accordance with the
orientation of the graph.
Then, we denote for $\yb \in e$ the unitary
tangent vector $\tb(\yb) := \gamb'(\gamb^{-1}(\yb))$, see Fig.~\ref{fig:example_pattern_Gamma_Y}
for two different examples. %

We will consider trails and {\cycles} (thus no repeating edges are allowed) that follow or may not follow the orientation of the graph.
If $\Gamma_Y$ is transformed to an undirected graph $\Gamma_Y'$ where each edge remains its parametrization,
then we denote by $\cT(\Gamma_Y')$ the set of all {\trails} in $\Gamma_Y'$.
Moreover, we define for each trail $T$ of the undirected graph $\Gamma_Y'$ of $\Gamma_Y$ the \emph{orientation function} $\chi_T: \Gamma_Y \to \{-1,0,+1\}$
that takes the value $\pm1$ on all points $\yb$ in an interior of an edge $e$ of $T$,
the sign indicating if the edge is passed according to the orientation of $e$ in the oriented graph $\Gamma_Y$ ($+1$) or not ($-1$),
and $0$ otherwise.

As the unit mesh $\Gamma_Y$ shall be connected in all directions to repetitions of itself
translated by $\eb_1 = (1,0)$ or $\eb_2 = (0,1)$
we assume that it touches each side of the unit square in such a way that
\begin{equation}\label{opposite}
  \exists \yb_1, \yb_2 \in \Gamma_Y : \yb_1 + \eb_1, \yb_2 + \eb_2 \in \Gamma_Y.
\end{equation}%
Note here that $\yb_1$ and $\yb_1 + \eb_1$ as well as $\yb_2$ and $\yb_2 + \eb_2$ are connected by
a trail of $\Gamma_Y'$ since $\Gamma_Y$ is connected.
This trail becomes a {\cycle} in the graph $\GYper$ obtained
by identifying
all such points on the sides of the square $[0,1]^2$,
\ie, $\yb$ with $\yb+\eb_1$ and $\yb$ with $\yb+\eb_2$
for $\yb \in \partial [0,1]^2$ -- and call them opposite points -- meaning that they correspond to the same vertex of $\cV$.
Moreover, all corner points $\yb \in \{(0,0),(1,0),(0,1),(1,1)\}\cap\Gamma_Y$, if they exist, are identified and represented by
one vertex in~$\cV$. %
The set of all {\cycles} in $\Gamma_{Y,\#}'$, the undirected version of the periodic graph $\Gamma_{Y,\#}$, we denote by $\cC(\Gamma_{Y,\#}')$.

Let the $\Omega$ be a rectangular domain of lengths $L_1$ and $L_2$,
such that the ratio $L_1 / L_2$ belongs to $\ZQ$. Then, we consider only \emph{periods}
$\delta > 0$ such that $N_1^\delta := L_1 / \delta$ and
$N_2^\delta := L_2 / \delta$ are integers
and 
refer as ``for any value of $\delta$''
if a statement holds for any such $\delta$.


Finally, we consider the plane mesh
$\Omegad$ depicted illustrative in Fig.~\ref{fig:example_Omega_delta} defined by
\begin{equation}
  \label{eq:lattice}
  \Omegad :=
  \bigcup_{n_1=0}^{N_1^\delta-1}
  \bigcup_{n_2=0}^{N_2^\delta-1} \Gamma_Y^\delta(n_1,n_2), \qquad
  \Gamma_Y^\delta(n_1,n_2) := \delta \big( \Gamma_Y + n_1 \eb_1 +
  n_2 \eb_2 \big)
\end{equation}
that can be described by a connected graph, where each vertex of $\Gamma^\delta$ corresponds to exactly one point in $\Gamma^\delta$,
\ie, there is no identification of points and so no periodic structure of $\Gamma^\delta$.
Note that opposite points in neighboring repetitions of $\Gamma^\delta_Y(n_1,n_2)$
coincide, making them the same vertex of $\Omegad$.

We introduce now the notion of derivatives and integrals on the meshes
$\Omegad$ and $\Gamma_Y$.

\begin{definition}[Derivative]
Let $e$ be an oriented edge with a $W^{1,\infty}$ and bijective arc length parametrization $\gamb$ that is in accordance with the orientation of $e$ and
let $u$ be a continuous function on $e$ such that $u \circ \gamb$ is differentiable.
We define its derivative, for all points $\xb \in e$, and $s = \gamb^{-1}(\xb)$ by
  \begin{equation}
    \label{eq:derivative}
\partial_\Gamma u(\xb) := \frac{\mathrm{d}}{\mathrm{d}s} \big( u \circ \gamb\big)(s)\ .
  \end{equation}
\end{definition}
Note that the derivative is orientation dependent.
However, due to the second order derivative in~\eqref{eq:Heat_equation} the solution $u^\delta$
and the variational problems will not be orientation dependent, so any fixed orientation will suit the setting.
Edge by edge we can extend this derivative on functions defined on graphs $\Gamma_Y$ and $\Omegad$.


\begin{definition}[Slow and fast derivative]
  \label{def:slow and fast derivative}
  Given a function $u$ on $\Omega \times \Gamma_Y$.
  If $u(\cdot,\yb)$ is differentiable on $\Omega$ for all $\yb \in \Gamma_Y$
  we define its \emph{slow derivative} by
  \begin{equation*}
\nabla_\xb u(\xb,\yb) \cdot    \tb(\yb)\,.
  \end{equation*}
Moreover, we define its \emph{fast derivative} $\partial_{\Gamma,\yb}$
  for all $\xb \in \Omega$ on which $u(\xb,\cdot)$ is differentiable on $\Gamma_Y$
  as the derivative $\partial_\Gamma$ with respect to $\yb$.
\end{definition}
For a function $u$ differentiable on $\Omega \times \Gamma_Y$, the function
$\xb \mapsto u\big(\xb,\tfrac{\xb}{\delta}\big)$ is differentiable on
$\Omegad$ and by the chain rule,
\begin{equation}
  \label{eq:deriv_on_mesh:micro_and_macro}
\partial_{\Gamma} \left(u\big(\xb,\tfrac{\xb}{\delta}\big)\right)
  = \nabla_\xb u \big(\xb,\tfrac{\xb}{\delta}\big) \cdot \tb(\tfrac{\xb}{\delta}) +
  \tfrac{1}{\delta} \partial_{\Gamma,\yb} u
  \big(\xb,\tfrac{\xb}{\delta}\big).
\end{equation}

Note that even for functions $u$ that are constant in $\yb$ (independent of $\yb$),
its slow derivative $\nabla_\xb   u(\xb) \cdot \tb(\yb)$ depends on the fast variable $\yb$ through the tangent field on $\Gamma_Y$.
This mean especially that for functions that are differentiable on $\Omega$ it holds for any $\delta > 0$
\begin{align}
  \label{eq:deriv_on_mesh:macro}
  \partial_\Gamma u(\xb) &= \nabla u(\xb)\cdot\tb(\tfrac{\xb}{\delta}) \quad\text{ on } \Gamma^\delta \subset \Omega\ .
\end{align}
For simplicity, also if $u$ is a function of $(t,\xb) \in (0,\tf)\times\Omega$
we use $\nabla u$ for the gradient of $u$ with respect to $\xb$.


\begin{definition}[Integral]
  Given a function $u$ defined on $\Gamma_Y$, we say that $u$ is
  integrable on $\Gamma_Y$ if $u$ is integrable on each edge $e \in
  \cE$. In that case, the integral of $u$ over $\Gamma_Y$ is given by
  \begin{equation*}
    \int_{\Gamma_Y} u(\yb) \,\mathrm{d}s(\yb) := \sum_{e \in \cE} \int_e
    u(\yb) \,\mathrm{d}s(\yb),
  \end{equation*}
  where the integral of $u$ over the edge $e$ is defined as the curve
  integral of the first kind, that is $\int_0^\ell u(\gamb(s))\,\mathrm{d}s$ for an arc length parametrization $\gamb: [0,\ell] \to e$.
\end{definition}

With this we can define in particular the length $\left|\Gamma_Y\right| := \int_{\Gamma_Y} \mathrm{d}s(\yb)$ of $\Gamma_Y$,
and the length $|\Gamma^\delta| =|\Omega| |\Gamma_Y|/\delta$ of~$\Gamma^\delta$ is defined for any  $\delta > 0$ where
$|\Omega|$ is the area of $\Omega = [0,L_1]\times [0,L_2]$.

For an oriented graph $\cG$ we denote by $\Ltwo(\cG)$ the Banach space of
functions $u$ such that $|u|^2$ is integrable on~$\cG$
and by $\Hone(\cG)$ the Hilbert space
\begin{align*}
  \Hone(\cG) &= \left\{u \in \Ltwo(\cG): \partial_{\Gamma} u
  \in \Ltwo(\cG), u \text{ continuous on } \cG \right\},
\end{align*}
the latter equipped with the norm
\begin{align*}
  \|v\|_{\Hone(\cG)}^2 &= \|v\|_{\Ltwo(\cG)}^2 +
  \|\partial_{\Gamma} v\|_{\Ltwo(\cG)}^2.
\end{align*}
In particular we will need to consider the spaces $\Ltwo(\Gamma_Y), \Hone(\Gamma_Y)$ and $\Hone(T)$ for any trail $T$ of $\Gamma'_Y$.
Note that orientation on a graph has to be chosen in order to define derivatives but there is no difference between $\Ltwo(\cG)$ and $\Ltwo(\cG')$ and $\Hone(\cG)$ and $\Hone(\cG')$, when $\cG'$ is the undirected graph of $\cG$.

As $\Gamma_{Y,\#}$ is a periodic manifold of $[0,1]^2$ the continuity of
$u \in \Hper(\Gamma_Y) := \Hone(\Gamma_{Y,\#})$ means that $u(\yb_1) = u(\yb_2)$ whenever $\yb_1$ and
$\yb_2$ correspond to the same vertex of $\Gamma_Y$, in particular in the
opposite points.
Note, that $H^1(T)$ for {\trails} $T$ of $\Gamma_Y'$
allows for functions that are discontinuous at the opposite points.

Let $\Gamma_D \subset \partial \Omega$ be a part of the boundary of
$\Omega$ that corresponds to a finite union of line segments of positive
length. For any $\delta$, let us define $\OmegadD :=
\partial {\Omegad} \cap \Gamma_D$. By assumption on $\Gamma_D$,
$\OmegadD$ is different from the empty set for sufficiently small $\delta$. We then
define the space $\cH^\delta$ as the subspace of functions $v \in
\Hone(\Omegad)$ whose trace vanishes on $\OmegadD$, \ie
\begin{equation*}
  \cH^\delta := \left\{ v \in \Hone(\Omegad) : v = 0 \mbox{ on } \OmegadD \right\}.
\end{equation*}
We consider now the following problem given in the weak formulation for a time
$\tf > 0$, a source term $f^\delta \in \Ltwo(0,\tf; \Ltwo(\Omegad))$ and an initial data
$u_\init \in \cH^\delta$: 
find $u^\delta$ in $\Ltwo(0,\tf;\cH^\delta)$ with $\partial_t u^\delta$ in $\Ltwo(0,\tf; \Ltwo(\Omegad))$ such that for all $v^\delta \in \cH^\delta$ and a.e. $t\in (0,\tf)$
\begin{subequations}
 \label{deltaP}
  \begin{align}
    \label{deltaP:1}
  \int_{\Omegad} \hspace{-0.1em}\rho c_p \partial_t{u}^\delta(t,\xb) v^\delta(\xb) \,\mathrm{d}s(\xb)
  + \int_{\Omegad} \hspace{-0.1em}a(\xb) \partial_{\Gamma} u^\delta(t,\xb)
  \partial_{\Gamma} v^\delta(\xb) \,\mathrm{d}s(\xb) &=
  \int_{\Omegad} \hspace{-0.1em}
  f^\delta(t,\xb) v^\delta(\xb) \,\mathrm{d}s(\xb),\\
   u^\delta(0,\xb) &= u_\init(\xb), \xb \in \Omegad.
  \end{align}
\end{subequations}
Here, $\rho$ is the mass density of the material and $c_p$ is the specific
heat capacity that are considered as constants,
$u_\init:\Omegad\to \ZR$ is a given initial temperature and the thermal conductivity
$a:\overline{\Omega} \to \ZR$ is a positive continuous function.
Thus
there exist
\begin{equation}\label{Aa}
  a_{\min} = \min_{\xb\in \Omega} a(\xb) > 0.
\end{equation}
It implies
\begin{equation*}
  a_{\min} \leqslant a(\xb) \leqslant \| a \|_{\Linf(\Omega)},
  \quad \xb \in \Omegad.
\end{equation*}
The problem \eqref{deltaP} is the weak formulation of the heat conduction problem~\eqref{eq:Heat_equation}
with heat source~$f^\delta$, initial condition given by $u_\init$,
homogeneous Dirichlet conditions on $\Gamma^\delta_D$ and
homogeneous Neumann conditions on the remaining part of the boundary.
Also note that the solution $u^\delta$ of the variational formulation~\eqref{deltaP} is independent of the
orientation on the graph even so the definition of the derivative
$\partial_{\Gamma}$
is orientation dependent (with difference only
in sign) since the curve integrals of the first kind
are orientation independent and  in the second term in the left
hand side of~\eqref{deltaP} the opposite signs cancel out.

Following the notations of \cite{bouchitte2001homogenization}, we
  introduce $\mu(\yb) := \tfrac{1}{|\Gamma_Y|} \mathrm{d}s(\yb)$ as the
  normalized periodic one-dimensional Hausdorff measure on $\Gamma_Y$,
  therefore the measure associated to $\Omegad$ is just given as $\mud(\xb) =
  \delta \mu(\xb / \delta)$. As
  $\delta \to 0$, the domain $\Omegad$ tends to a dense subdomain of $\Omega$, so
  that the natural idea is to consider a limit problem of~(\ref{deltaP}) on
  $\Omega$. Following the two-scale homogenization
  theory~\cite{bouchitte2001homogenization,AllaireDamlamianHornung,RaduDiploma},
  the measure that has to be considered for the limit case is not the usual
  Lebesgue measure $\cL^2 \llcorner \Omega$ on $\Omega$ but the tensorial
  product $\cL^2 \llcorner \Omega \bigotimes \mu \llcorner \Gamma_Y$.  

To stress the passage from the 1D mesh to the 2D domain we call the associated two-scale
convergence in this paper the \emph{mesh two-scale convergence}. The convergence is given for functions
defined on a mesh and was already introduced in the literature, see \eg
\cite{AllaireDamlamianHornung, NeussRadu}. However, since we have the time $t$
as an additional parameter we consider the mesh two-scale
convergence with a parameter where we follow the work of Neuss-Radu~\cite{NeussRadu,
  RaduDiploma}.

\begin{definition}[Mesh two-scale convergence]
    \label{dmts} We say that sequence of functions
    $(v^\delta) \subset \Ltwo(0,\tf;\Ltwo(\Omegad))$ with $\delta>0$ for $\delta\to0$ is "mesh
    two-scale convergent" with limit $v^0 \in \Ltwo(0,\tf;\Loper)$ if for
    each
    $\psi \in C^\infty([0,\tf]; C^{\infty}(\Omega; \Hper(\Gamma_Y) ))$ we have
    \begin{equation}\label{vdelta}
    \lim_{\delta \to 0} \delta \int_0^{\tf} \int_{\Omegad} v^\delta (t, \xb) \psi \left(t, \xb,
    \frac{\xb}{\delta}\right) \,\mathrm{d}s(\xb) \, \mathrm{d}t= \int_0^{\tf} \int_{\Omega}
    \int_{\Gamma_Y}v^0(t, \xb, \yb) \psi (t, \xb, \yb) \,\mathrm{d}s(\yb) \,\mathrm{d}\xb \, \mathrm{d}t.
    \end{equation}
    In this case we use the abbreviation $v^\delta \xrightharpoonup{\mathrm{m2s}} v^0$.
\end{definition}

Note that with a simple scaling argument we have that for any $\delta >0$
\begin{equation*}
  \delta \| \mathds{1}_{\Omegad} \|_{\Ltwo(\Omegad)}^2 = \| \mathds{1}_{\Omega
    \times \Gamma_Y} \|_{\Ltwo(\Omega \times \Gamma_Y)}^2,
\end{equation*}
where $\mathds{1}_X$ is the characteristic function of a set $X$. This
explains the factor $\delta$ in the left hand side of the mesh two-scale
convergence and the factor $\sqrt{\delta}$ in the various $\Ltwo$-norm estimates that will follow.


Having in mind the mesh two-scale convergence we consider $f^\delta$ in (\ref{deltaP}) to
depend on the fast variable as well, \ie, there exists a function $f$ on
$(0,\tf) \times \Omega \times \Gamma_Y$ such that $f^\delta(t,\xb) = f(t,\xb,\frac{\xb}{\delta})$
on $(0,\tf) \times \Gamma^\delta$ for all $\delta > 0$. %
In the following lemma we give uniform bounds on the data of the $\delta$-dependent problem.
They are then used to show its well-posedness and {\em a priori} stability estimates on
its solution $u^\delta$ as stated in Theorem~\ref{theo:wellposedness}.
The {\em a priori} estimates will be needed in the proof of the mesh two-scale convergence the homogenized model (Theorem~\ref{theo:convergence}) in Sec.~\ref{sec:homogenized_model:convergence}
and of the convergence in norm (Theorem~\ref{theo:convergence_in_norm}) in Sec.~\ref{sec:homogenized_model:convergence_in_norm}.

\begin{lemma}[Uniform bound of the data]
  \label{lema:bound_data}
  \begin{enumerate}
\item[a)]  Let $f\in \Ltwo(0,\tf;\Ltwo(\Gamma_Y;C(\overline{\Omega})))$ and
  $f^\delta (t,\xb) = f(t,\xb,\frac{\xb}{\delta})$ on
  $(0,\tf)  \times \Gamma^\delta$ for $\delta > 0$.
  Then, for any $\delta >0$ it hold that
  $f^\delta\in \Ltwo(0,\tf ; \Ltwo(\Omegad))$ and
  \begin{equation}
    \label{eq:estimate_f}
    \delta \bnorm{f^\delta}{\Ltwo(0,\tf ; \Ltwo(\Omegad))}^2 \leqslant
    |\Omega| \bnorm{f}{\Ltwo(0,\tf ; \Ltwo(\Gamma_Y; \Linf(\Omega)))}^2.
  \end{equation}
\item[b)]
  If $u_\init \in C^1(\overline{\Omega})$ such that
  $u_\init = 0$ on $\Gamma_D$. %
  Then, for any $\delta > 0$, $u_\init\in \cH^\delta$ and
  \begin{equation}
    \label{eq:estimate_u_init}
 \delta \bnorm{u_\init}{\Hone(\Omegad)}^2 \leqslant
    |\Omega| (1 + 2 |\Gamma_Y|) \bnorm{u_\init}{W^{1,\infty}(\Omega)}^2.
  \end{equation}
  \end{enumerate}
\end{lemma}
\begin{proof}
  From the definition of norm we get
  \begin{equation*}
    \aligned
    \delta \bnorm{f^\delta}{\Ltwo(0,\tf ; \Ltwo(\Omegad))}^2 &=\delta
    \sum_{n_1, n_2} \int_0^{\tf} \delta \int_{\Gamma_Y} |f(t, \delta(\yb +
    n_1 \eb_1 +n_2 \eb_2), \yb)|^2 ds(\yb)\, \mathrm{d}t\\
    &\leqslant \delta^2 \sum_{n_1, n_2} \int_0^{\tf} \int_{\Gamma_Y} \xnorm{f(t, \cdot, \yb)}{\Linf(\Omega)}^2 ds(\yb)\, \mathrm{d}t
    \leqslant |\Omega| \bnorm{f}{\Ltwo(0,\tf ; \Ltwo(\Gamma_Y; \Linf(\Omega)))}^2,
    \endaligned
  \end{equation*}
and, hence, $f^\delta \in \Ltwo(0,\tf ; \Ltwo(\Omegad))$ for any $\delta > 0$.
  Just in the same way we obtain
  \begin{equation*}
    \delta \bnorm{u_\init}{L^2(\Omegad)}^2 \leqslant
    |\Omega| \bnorm{u_\init}{L^\infty(\Omega)}^2,
  \end{equation*}
  and similarly
  \begin{equation*}
  \aligned
  \delta \xnorm{\partial_{\Gamma} u_\init}{\Ltwo(\Omegad)}^2 
  &=\delta \sum_{n_1, n_2} \delta \int_{\Gamma_Y} \left|
    \nabla_\xb u_\init (t, \delta(\yb + n_1 \eb_1 +n_2 \eb_2)) \cdot
    \tb(\yb) \right|^2 ds(\yb)\\
  &\leq \delta^2 \sum_{n_1, n_2} \int_{\Gamma_Y}
  2\xnorm{\nabla_\xb u_\init (t, \cdot)}{\Linf(\Omega)}^2 ds(\yb)
  \leq 2 |\Omega| |\Gamma_Y |\xnorm{\nabla_\xb u_\init}{\Linf(\Omega)}^2.
  \endaligned
  \end{equation*}
  Hence, $u_\init \in \cH^\delta$ and the proof is complete. \qedhere
\end{proof}


\begin{lemma}\label{ldeltaexist}
  Let $\delta>0$, $f^\delta\in \Ltwo(0,\tf;\Ltwo(\Omegad))$ and
  $u_{\init} \in \cH^\delta$.  Then the problem~\eqref{deltaP}
  has a unique solution that satisfies
  $$
  u^\delta\in \Ltwo(0,\tf; \cH^\delta \cap H^2(\mathring{\Gamma}^\delta))
  \qquad \partial_t{u}^\delta \in
  \Ltwo(0,\tf;\Ltwo(\Omegad)),
  $$
where $\mathring{\Gamma}^\delta := \Gamma^\delta \backslash \cV$
  the graph without its vertices and $H^2(\mathring{\Gamma}^\delta)$
  the Sobolev space of functions that are in $H^2(e)$ on each edge $e$ of $\Gamma^\delta$.
\end{lemma}
\begin{proof}
The assumption on $f$ implies that for any $\delta > 0$ that
   $f^\delta \in L^2(0,\tf; L^2(\Omegad)) \subset L^2(0,\tf; (\cH^\delta)')$
   and $u_\init \in \cH^\delta \subset L^2(\Omegad)$. %
   Hence, the classical solution theory for parabolic equations
   gives the existence and uniqueness of the solution $u^\delta$ of~\eqref{deltaP}
  that satisfies
  $u^\delta\in \Ltwo(0,\tf;\cH^\delta), \partial_t{u}^\delta \in
  \Ltwo(0,\tf;(\cH^\delta)')$,
  see \eg~\cite[Theorem 1 in page 558]{DL}.

  With the regularity of $f^\delta$ and $u_\init$ we can apply classical regularity results to
  get that
  $u^\delta \in \Ltwo(0,\tf;\cH^\delta \cap H^2(\mathring{\Gamma}^\delta))$,
  and $ \partial_t{u}^\delta \in \Ltwo(0,\tf;\Ltwo(\Omegad))$, see \eg~\cite[Theorem 5 in Chapter 7]{Evans}.
\end{proof}

\begin{lemma}[\emph{A priori} stability estimates]\label{lapriori}
  Let for $\delta>0$
  $\sqrt{\delta} \| f^\delta \|_{\Ltwo(0,\tf;\Ltwo(\Omegad))}$,
  $\sqrt{\delta} \| u_\init \|_{\Ltwo(\Omegad)}$ and
  $\sqrt{\delta} \| \partial_{\Gamma} u_\init \|_{\Ltwo(\Omegad)}$ be
  uniformly bounded with respect to $\delta$, and let $u^\delta$ be the unique solution of (\ref{deltaP}). Then there is $C>0$ such
  that
  \begin{equation*}
    \begin{aligned}
      \sqrt{\delta} \|u^\delta\|_{\Ltwo(0,\tf;\Ltwo(\Omegad))} &
      \leqslant C,\\
      \sqrt{\delta} \|\partial_{\Gamma}
      u^\delta\|_{\Ltwo(0,\tf;\Ltwo(\Omegad))} & \leqslant C,\\
      \sqrt{\delta}
      \|\partial_t{u}^\delta\|_{\Ltwo(0,\tf;\Ltwo(\Omegad))} & \leqslant
      C.
    \end{aligned}
  \end{equation*}
\end{lemma}
\begin{proof}
  First testing for any $\tau \in (0,\tf)$
  the variational formulation (\ref{deltaP}) with $u^\delta(t,\xb)$ we find for any $t \in (0,\tau)$
    \begin{equation*}
      \frac{1}{2} \frac{\text{d}}{\text{d} t}
      \xnorm{u^\delta(t,\cdot)}{\Ltwo(\Omegad)}^2 +
      a_{\min{}} \xnorm{\partial_{\Gamma} u^\delta(t,\cdot)}{\Ltwo(\Omegad)}^2 \leq
      \int_{\Omega} f^\delta(t,\xb) u^\delta(t,\xb) \, \mathrm{d}\mud(\xb).
    \end{equation*}
    Then, applying the Cauchy-Schwartz inequality and the Young inequality we get
    \begin{equation*}
      \frac{1}{2} \frac{\text{d}}{\text{d} t}
      \xnorm{u^\delta(t,\cdot)}{\Ltwo(\Omegad)}^2 +
      a_{\min{}} \xnorm{\partial_{\Gamma} u^\delta(t,\cdot)}{\Ltwo(\Omegad)}^2
      \leqslant \frac{1}{4\eta}
      \xnorm{f^\delta(t,\cdot)}{\Ltwo(\Omegad)}^2 + \eta
      \xnorm{u^\delta(t,\cdot)}{\Ltwo(\Omegad)}^2.
    \end{equation*}
    Integrating this relation over $t \in (0,\tau)$ we obtain
    \begin{multline}
      \label{eq:apriori_estimates_proof1}
      \frac{1}{2} \xnorm{u^\delta(\tau,\cdot)}{\Ltwo(\Omegad)}^2 +
      a_{\min{}} \int_0^\tau \xnorm{\partial_{\Gamma}
        u^\delta(t,\cdot)}{\Ltwo(\Omegad)}^2 \, \mathrm{d}t \\
      \leqslant \frac{1}{4 \eta} \int_0^\tau
      \xnorm{f^\delta(t,\cdot)}{\Ltwo(\Omegad)}^2 \, \mathrm{d}t +
      \eta \int_0^\tau \xnorm{u^\delta(t,\cdot)}{\Ltwo(\Omegad)}^2 \,
      \mathrm{d}t + \frac{1}{2}
      \xnorm{u_\init}{\Ltwo(\Omegad)}^2.
  \end{multline}
  We take then $\eta = 1 / (4\tf)$ such that
  \begin{equation}
    \label{eq:apriori_estimates_assumption_data}
    \eta \int_0^\tau \xnorm{u^\delta(t,\cdot)}{\Ltwo(\Omegad)}^2 \leqslant
    \frac{1}{4}
    \sup_{t \in (0,\tf)}
    \xnorm{u^\delta(t,\cdot)}{\Hone(\Omegad)}^2,
  \end{equation}
  so that relation~(\ref{eq:apriori_estimates_proof1}) gives
  \begin{multline}
    \label{eq:apriori_estimates_proof2}
    \frac{1}{2} \xnorm{u^\delta(\tau,\cdot)}{\Ltwo(\Omegad)}^2 +
    a_{\min{}} \int_0^\tau \xnorm{\partial_{\Gamma}
      u^\delta(t,\cdot)}{\Ltwo(\Omegad)}^2 \, \mathrm{d}t \\
    \leqslant \tf \int_0^{\tf}
    \xnorm{f^\delta(t,\cdot)}{\Ltwo(\Omegad)}^2 \, \mathrm{d}t +
    \frac{1}{4}    \sup_{t \in (0,\tf)}
    \xnorm{u^\delta(t,\cdot)}{\Ltwo(\Omegad)}^2 + \frac{1}{2}
    \xnorm{u_\init}{\Ltwo(\Omegad)}^2.
  \end{multline}
  Now, by assumption there exists a constant $D$ such that, for any
  $\delta$,
  \begin{equation*}
    \delta \int_0^{\tf} \xnorm{f^\delta(t,\cdot)}{\Ltwo(\Omegad)}^2 \,
    \mathrm{d}t \leqslant D \quad \text{and} \quad
    \delta \xnorm{u_\init}{\Ltwo(\Omegad)}^2 \leqslant D.
  \end{equation*}
  \begin{enumerate}
  \item Relation~(\ref{eq:apriori_estimates_proof2}) implies
    \begin{equation*}
      \delta  \xnorm{u^\delta(\tau,\cdot)}{\Ltwo(\Omegad)}^2
      \leqslant (2\tf+1)D + \delta \frac{1}{2} \sup_{t \in (0,\tf)}
      \xnorm{u^\delta(t,\cdot)}{\Ltwo(\Omegad)}^2\ ,
    \end{equation*}
    that gives
    \begin{equation}\label{ineq}
      \delta \sup_{t \in (0,\tf)}
      \xnorm{u^\delta(t,\cdot)}{\Ltwo(\Omegad)}^2 \leqslant (4\tf+2)D,
    \end{equation}
    and so we find
    \begin{equation*}
      \delta \|u^\delta\|_{\Ltwo(0,\tf;\Ltwo(\Omegad))}^2 \leqslant \tf(4\tf+2)D.
    \end{equation*}
  \item Using again relation~(\ref{eq:apriori_estimates_proof2}) with (\ref{ineq}) gives
    \begin{equation*}
      \delta \int_0^{\tf} \xnorm{\partial_{\Gamma}
        u^\delta(t,\cdot)}{\Ltwo(\Omegad)}^2 \, \mathrm{d}t
      \leqslant \frac{1}{a_{\min{}}}\left(\tf D+\frac{1}{4}(4\tf+2)D+\frac{1}{2}D\right)=\frac{1}{a_{\min{}}}(2\tf+1)D.
    \end{equation*}
  \end{enumerate}

  Similarly to the derivation of~(\ref{eq:apriori_estimates_proof2})
  and starting from the variational
  formulation (\ref{deltaP}) with the particular
  test function $\partial_t u^\delta(t,\xb)$, it holds
  \begin{multline*}
    \xnorm{\partial_t u^\delta}{\Ltwo((0,\tf); \Ltwo(\Omegad))}^2 +
    \frac{a_{\min{}}}{2}
    \xnorm{\partial_{\Gamma} u^\delta(\tf,\cdot)}{\Ltwo(\Omegad)}^2 \\
    \leqslant \frac{1}{4\eta} \xnorm{f}{\Ltwo((0,\tf);
      \Ltwo(\Omegad))}^2 + \eta \xnorm{\partial_t
      u^\delta}{\Ltwo((0,\tf); \Ltwo(\Omegad))}^2 +
    \frac{\|a\|_{L^\infty(\Omega)}}{2}\xnorm{\partial_{\Gamma} u_\init}{\Ltwo(\Omegad)}^2.
  \end{multline*}
  Choosing $\eta=1/2$ leads to uniform estimate for
  $\delta \xnorm{\partial_t u^\delta}{\Ltwo((0,\tf);
    \Ltwo(\Omegad))}^2$.
\end{proof}

Now, we give the existence and uniqueness and {\em a-priori} stability estimates for the problem~(\ref{deltaP}).  It is the classical heat equation, but on a one-dimensional mesh--like domain.  The existence
and uniqueness result for the equation follows by the standard
arguments of linear second-order parabolic equations with the addition
of mild regularity results applied on mesh edges separately.

\begin{theorem}[Well-posedness of the $\delta$-dependent problem]
  \label{theo:wellposedness}
  Let $f^\delta (t,\xb) = f(t,\xb,\frac{\xb}{\delta})$ on
  $(0,\tf)  \times \Gamma^\delta$ for all $\delta > 0$ with
  $f\in \Ltwo(0,\tf;\Ltwo(\Gamma_Y;C(\overline{\Omega})))$
  and $u_\init \in C^1(\overline{\Omega})$ with $u_\init = 0$ on $\Gamma_D$.
  %
  Then, problem~\eqref{deltaP} has a unique solution
  $u^\delta\in \Ltwo(0,\tf; \cH^\delta \cap H^2(\mathring{\Gamma}^\delta))$ with $\partial_t{u}^\delta \in
  \Ltwo(0,\tf;\Ltwo(\Omegad))$ and
  there is $C>0$ such that for all $\delta > 0$
  \begin{equation}
      \sqrt{\delta} \left(%
        \|u^\delta\|_{\Ltwo(0,\tf;\Ltwo(\Omegad))} +
        \|\partial_{\Gamma} u^\delta\|_{\Ltwo(0,\tf;\Ltwo(\Omegad))} +
        \|\partial_t{u}^\delta\|_{\Ltwo(0,\tf;\Ltwo(\Omegad))}\right) \leqslant C.
        \label{eq:stability}
  \end{equation}
\end{theorem}

\begin{proof}
   The assumptions on $f$ and $u_\init$ imply by Lemma~\ref{lema:bound_data}
   for any $\delta > 0$ that
   $f^\delta \in L^2(0,\tf; L^2(\Omegad))$.
   Hence by Lemma~\ref{ldeltaexist} the solution of \eqref{deltaP} exists and satisfies
   $u^\delta\in \Ltwo(0,\tf;\cH^\delta \cap H^2(\mathring{\Gamma}^\delta))$ and  $\partial_t{u}^\delta \in
   \Ltwo(0,\tf;\Ltwo(\Omegad))$.
   As with Lemma~\ref{lema:bound_data} the assumptions of Lemma~\ref{lapriori} are fulfilled we conclude the {\em a priori} estimate \eqref{eq:stability}.
\end{proof}

In the following two theorems the main results of the paper are stated, namely
\begin{itemize}
 \item the mesh two-scale convergence $u^\delta$ to the solution $u^0$ of the homogenized problem in Theorem~\ref{theo:convergence}
 that will be proved in Sec.~\ref{sec:homogenized_model:convergence}, where the properties of the homogenized model
 will be shown in Sec.~\ref{sec:homogenized_model:properties}, and
 \item under additional regularity assumption on the data the convergence in norm in Theorem~\ref{theo:convergence_in_norm}
 that will be proved in Sec.~\ref{sec:homogenized_model:convergence_in_norm}. %
\end{itemize}

\begin{theorem}[Mesh two-scale convergence to the homogenized problem]
  \label{theo:convergence}
  Let $\Gamma_Y$ be connected and satisfies \eqref{opposite} and
  let the assumptions of Theorem~\ref{theo:wellposedness} be fulfilled.
  Then the sequence $(u^\delta)_{\delta>0}$ of solutions of \eqref{deltaP} for $\delta \to 0$
   mesh two-scale converges to the function $u^0$ which is the
  unique solution of the problem: find
  \begin{equation*}
    u^0 \in \Ltwo(0,\tf;\cH) \quad  \mbox{ such that } \quad
    \partial_t{u}^0\in \Ltwo(0,\tf;\Ltwo(\Omega)),
  \end{equation*}
  where $\cH = \{ v \in \Hone(\Omega): v = 0 \text{ on } \Gamma_D\}$,
  such that for all $v \in \cH$
  \begin{subequations}
  \label{P}
  \begin{align}
    \label{P:PDE}
    \int_\Omega  \rho c_p \partial_t{u}^0 (t,\xb) v(\xb) \mathrm{d}\xb +
    \int_{\Omega} a(\xb)  \Abbh \nabla u^0(t, \xb) \cdot \nabla
    v(\xb) \mathrm{d}\xb &= \int_{\Omega}  {f_\homg}(t, \xb)     v(\xb) \mathrm{d}\xb, \\
    u^0(0,\xb) &= u_\init(\xb), \qquad \xb \in \Omega,
    \label{P:IC}
  \end{align}
  \end{subequations}
  where
  \begin{equation}
     \label{eq:f_hom}
     f_\homg :=  \frac{1}{|\Gamma_Y|} \int_{\Gamma_Y} f(\cdot,\cdot,\yb) \, \mathrm{d}s{(\yb)}
  \end{equation}
  and where the constant 2-by-2 symmetric and positive definite matrix $\Abbh$ is given by
  \begin{equation}
    \label{eq:A_hom}
    \Abbh:=  \frac{1}{|\Gamma_Y|} \int_{\Gamma_Y} \left(\tb(\yb) +
      \partial_{\Gamma} \phib(\yb)
    \right) \left( \tb(\yb) +  \partial_{\Gamma} \phib(\yb)\right)^T
    \, \mathrm{d}s(\yb),
  \end{equation}
  where $\phib \in \Hper(\Gamma_Y)^2$ is an arbitrary solution of the canonical problem
\begin{equation}
    \tag{\ensuremath{\cP_\phib}}
    \label{canon}
   \int_{\Gamma_Y} \partial_{\Gamma} \phib(\yb)
  \partial_{\Gamma} \psi(\yb) \,\mathrm{d}s(\yb) = -
  \int_{\Gamma_Y} \tb(\yb) \partial_{\Gamma} \psi(\yb)
   \,\mathrm{d}s(\yb), \qquad \forall\psi \in \Hper(\Gamma_Y).
 \end{equation}
Moreover, $u^\delta|_{t=0}$ mesh two-scale converges to the initial data $u_\init$ of~\eqref{P}.
\end{theorem}

Solvability of the problem \eqref{canon} will be stated in
Lemma~\ref{lcanon}, the symmetry and positive definiteness of the homogenized tensor
$\Abbh$ in Lemma~\ref{lA} and the existence and uniqueness of the
problem (\ref{P}) in Lemma~\ref{lmodel}.

\begin{remark}
  We have already noted that $u^\delta$ is orientation
  independent. Furthermore since $\Abbh$ is with its definition~\eqref{eq:A_hom}
  orientation independent the same holds for $u^0$.
\end{remark}

\begin{theorem}[Convergence to the homogenized problem in norm]
  \label{theo:convergence_in_norm}
Let the assumption of Theorem~\ref{theo:wellposedness} be fulfilled and let in addition $u_\init \in H^3(\Omega)$ and
  \begin{equation*}
    f\in
    \Ltwo(0,\tf;\Ltwo (\Gamma_Y;\Htwo(\Omega))) \ \mbox{ s.t. }
    \  \partial_t f \in \Ltwo(0,\tf;\Ltwo(\Gamma_Y;\Ltwo(\Omega)))\ .
  \end{equation*}
  Then it holds in the limit $\delta\to 0$
  \begin{equation}
     \label{eq:convergence_norms}
    \sqrt{\delta}\|u^\delta-u^0\|_{C([0,\tf];\Ltwo(\Omegad))} \to
    0, \qquad \sqrt{\delta}\|\partial_{\Gamma} (u^\delta -
    u^0-\delta u^1)\|_{\Ltwo(0,\tf;\Ltwo(\Omegad))} \to 0,
  \end{equation}
  where $u^0$ is the solution of~\eqref{P}, $u^1(t,\xb,\yb) = \nabla u^0(t,\xb) \cdot \phib(\yb)$
  and $\phib$ is an arbitrary solution of the canonical problem~\eqref{canon}.
\end{theorem}

The convergence to the limit solution in $L^2(\Gamma^\delta)$, where for a convergence of the gradients
in~$L^2(\Gamma^\delta)$ the first order corrector has to be added, is in accordance with the usual homogenization
of elliptic and parabolic equations with periodic pattern~\cite{CD}.

\section{The homogenized model}
\label{sec:homogenized_model}


\subsection{Technical preliminaries}
\label{sec:tp}

\begin{lemma} \label{LemaIntegralFirstKind} Let $\cG$ be a oriented graph that
  may be periodic or not where each of its edges admits a $W^{1,\infty}$ and
  bijective arc length parametrization.  Let, moreover, $T$ be a trail of
  $\mathcal{G}'$ from $\xb_0$ to $\xb_1$, and let $\varphi \in
  \Hone(T)$. Then,
  \begin{equation}
    \label{lemma:integration_fprime_trail}
    \int_{ T } \partial_\Gamma \varphi (\xb) \chi_T (\xb) \mathrm{d} s(\xb) = \varphi(\xb_1) - \varphi(\xb_0)\,.
  \end{equation}
\end{lemma}
\begin{proof}
As $T$ represents a continuous and piecewise differentiable curve
  the lemma is a direct consequence of the definition of the derivative~\eqref{eq:derivative} and the
  fundamental theorem of calculus.
\end{proof}
As simple consequence of the orientation dependence of the derivative and Lemma~\ref{LemaIntegralFirstKind}
it holds for any \cycle~$C \in \cC(\Gamma_{Y,\#}')$
and any function $\varphi \in \Hper(\Gamma_Y)$ that
\begin{equation}
  \label{eq:integration_fprime_cycle:2}
  \int_{ \Gamma_Y } \partial_\Gamma \varphi (\xb) \chi_C(\xb)\mathrm{d} s(\xb) = 0\,.
\end{equation}
Note that the formula (\ref{lemma:integration_fprime_trail}) also holds for any walk in $\cG'$ and consequently formula (\ref{eq:integration_fprime_cycle:2}) holds for closed walks.

A form of reciprocal of \eqref{eq:integration_fprime_cycle:2} is true as well: if
the integral of a function $\psi$ along any {\cycle} $C \in \cC(\Gamma_{Y,\#}')$ multiplied with the orientation function $\chi_C$ is equal to zero, then
$\psi$ possesses a potential $\varphi \in \Hper(\Gamma_Y)$.
\begin{lemma} \label{RemarkAntider}
  Let for $\psi \in \Ltwo(\Gamma_Y)$ and all {\cycles} $C \in \cC(\Gamma_{Y,\#}')$
  \begin{equation*}
    \int_{ \Gamma_Y } \psi(\yb) \chi_C(\yb) \, \mathrm{d}s(\yb) = 0\,.
  \end{equation*}
  Then there exists $\varphi \in \Hper(\Gamma_Y)$ such that $\partial_{\Gamma} \varphi=\psi$ almost everywhere on $\Gamma_Y$.
\end{lemma}

\begin{proof}
We start defining $\varphi(\yb_0) = 0$ on an arbitrary vertex $\yb_0 \in \Gamma_Y$.

   Then, for any other vertex $\yb_1 \in \Gamma_Y$ we fix a trail $T$ of $\Gamma_Y'$ from $\yb_0$ to $\yb_1$ and define
   \begin{align*}
      \varphi(\yb_1) = \int_T \psi(\yb) \chi_T(\yb) \text{d}s(\yb)\ .
   \end{align*}
   As for any other trail $T'$ of $\Gamma_Y'$ from $\yb_0$ to $\yb_1$
   it holds by assumption of the lemma that for the closed walk $C$ concatenating $T$ and $T'_-$, the inverse trail to $T'$,
   \begin{align*}
      0 = \int_{T} \psi(\yb) \chi_T(\yb) \text{d}s(\yb) + \int_{T'_-} \psi(\yb) \chi_{T'_-}(\yb) \text{d}s(\yb) \ .
   \end{align*}
   As
   $\int_{T'_-} \psi(\yb) \chi_{T'_-}(\yb) \text{d}s(\yb) = \int_{T'} \psi(\yb)
   (-\chi_{T'}(\yb)) \text{d}s(\yb)$
   the definition of $\varphi(\yb_1)$ is independent of the choice of
   the trail.

   Finally, for all points $\yb$ in the interior of an edge $e \in \cE$ with
   arc length parametrization $\gamb$ let $\ell = \gamb^{-1}(\yb)$ and
   $\varphi(\yb) = \varphi(\gamb(0)) + \int_0^{\ell} \psi(\gamb(s)) \text{d}s$,
   which gives a continuous definition on its starting and end point of $e$
   and, hence, due to the trail independence on all vertices of $\cV$.  With
   the definition of the derivative~\eqref{eq:derivative} and the fundamental
   theorem of calculus it holds
   $\partial_\Gamma \varphi = \psi$ on each edge and so
   $\varphi \in \Hper(\Gamma_Y)$. This completes the proof.
\end{proof}

\begin{lemma}\label{LemaPhiConst}
  Let $\cG$ be strongly connected graph and for $v \in \Ltwo(\cG)$, all $\psi\in \Hone(\cG)$ and all {\cycles} $C \in \cC(\cG')$
  it holds
  \begin{equation}
    \int_{ \cG } v(\yb) \partial_\Gamma \psi (\yb)\chi_C(\yb) \, \mathrm{d}s(\yb) = 0\,.
    \label{eq:LemaPhiConst}
  \end{equation}
  Then, $v$ is constant on $\cG$.
\end{lemma}
\begin{proof}
   For a strongly connected graph $\cG$ any edge $e \in \cE(\cG)$ belongs to at least one \cycle~$C \in \cC(\cG')$.
   Restricting to test functions with support only on $e$ we find
   that
   \begin{equation*}
    \int_{ e} v(\yb) \partial_\Gamma \psi (\yb) \, 
    \mathrm{d}s(\yb) = 0\, \quad\forall \psi \in C^\infty_0(e)\ .
  \end{equation*}
  This implies that $v$ takes constant values on all edges $e \in \cE(\cG)$.

  Now, let $e_1$ and $e_2$ be two edges of $\cG$ with a common vertex
  $\xb \in \cE(e_1) \cap \cE(e_2)$. By strong connectivity of $\cG$ these two edges belong to some
  \cycle~$C$. Then, restricting to test functions supported on
  $\widetilde{\Gamma}:= e_1 \cup e_2 \cup \lbrace \xb \rbrace$ the
  equality~\eqref{eq:LemaPhiConst} implies
  \begin{equation*}
     0 = \int_{e_1 \cup e_2} v(\yb) \partial_\Gamma \psi(\yb) \chi_C(\yb) \, \mathrm{d}s(\yb)\, \quad\forall \psi \in H^1_0(\widetilde{\Gamma}).
  \end{equation*}
  Then, integrating by parts and using that $\partial_\Gamma v = 0$ on $e_1$ and $e_2$ we find
  \begin{equation*}
     0 = \left(v|_{e_1} - v |_{e_2}\right) \psi(\xb)\, \quad\forall \psi(\xb) \in \ZR\ .
  \end{equation*}
  As $v$ takes the same constant value on two arbitrary neighbouring edges it is constant on the whole~$\cG$.
\end{proof}

\begin{lemma}\label{LemaPhiIbyP}
  For $v \in \Honezero(\Omegad)$ and any {\cycle} $C \in
  \mathcal{C}(\Gamma'_{Y,\#})$ it holds
  \begin{equation*}
    \int_{\Omegad} \partial_{\Gamma} v \left(\xb\right) \chi_C
    \left(\frac{\xb}{\delta}
    \right)  \, \mathrm{d} s (\xb) = 0.
  \end{equation*}
\end{lemma}

\begin{proof}
  We decompose the support of $\chi_C\big( \tfrac{\xb}{\delta}\big)$ into the
  union of {\cycles} $(C_i) \in \cC(\Omegad')$ and trails $(T_j) \in \cT(\Omegad')$, where each
  trail $T_j$ is of length $\ell_j$, respectively, and admits a parametrization $\gamma_j$ such
  that $\gamma_j(0)$ and $\gamma_j(\ell_j)$ belongs to $\partial \Omega$ and so
  $v(\gamma_j(0)) = v(\gamma_j(\ell_j)) = 0$. Using then
  Lemma~\ref{LemaIntegralFirstKind} on each {\cycle} $C_i$ and on each trail $T_j$
  leads to the desired result.
\end{proof}

The following theorem is the key theorem in the classical homogenization
theory transferred to the considered setting of periodic meshes.
Its proof follows exactly the one in the usual theory (see~\cite[Theorem 1.5.5]{NeussRadu}).
\begin{theorem}\label{tcomp1}
  Let the sequence $(v^\delta) \subseteq \Ltwo(0,\tf; \Ltwo(\Omegad))$
  with $\delta > 0$ and $\delta\to0$ be such that there is $C>0$ and it holds
  for all $\delta>0$ that
  \begin{equation*}
    \sqrt{\delta} \|v^\delta\|_{\Ltwo(0,\tf;\Ltwo(\Omegad))} \leqslant C.
  \end{equation*}
  Then there is a subsequence that we again denote by $(v^\delta)$
  and a function $v^0 \in \Ltwo(0,\tf;\Loper)$ such that $v^\delta$ mesh two-scale converges to~$v^0$.
\end{theorem}

\begin{remark} The time-dependent mesh two-scale convergence implies \emph{a
    posteriori} the usual mesh two-scale convergence, since
  Theorem~\ref{tcomp1} is also valid when the family $(v^\delta)$ is
  independent of time.
  \end{remark}

Note, that Definition~\ref{dmts} of the mesh two-scale convergence uses highly smooth test
  functions, but it can be generalized to functions with less
  regularity.
  Especially, test functions in  $\Ltwo(0,\tf; \Ltwo(\Gamma_Y; C(\Omega))$,
  which we call admissible functions, can be considered.

\subsection{Properties of the homogenized model}
\label{sec:homogenized_model:properties}

\begin{lemma}\label{lcanon}
  The problem~\eqref{canon} has a unique solution up to an additive constant in $\ZR^2$.
\end{lemma}
\begin{proof}
  First note that if $\phib$ is a solution of~\eqref{canon}, then for
  any constant $\Cb \in \ZR^2$, $\phib + \Cb$ is also a solution
  of~\eqref{canon}. Therefore, we consider the two equations of~\eqref{canon} separately and restrict the problem to the periodic space with vanishing average
    $\Hperaver := \big\lbrace \theta \in \Hper(\Gamma_Y)
    \text{ such that }
    \int_{\Gamma_Y} \theta(\yb) \, \mathrm{d}\yb = 0 \big\rbrace$,
  on which the semi-norm
  $\| \partial_{\Gamma} \theta \|_{\Ltwo(\Gamma_Y)}$ is a norm.
  Problem~\eqref{canon} becomes then
  a classical elliptic problem and its well-posedness follows from
the Lax-Milgram theorem.   Note that no compatibility condition, that
is sometimes called necessary condition,
is needed due to  the special right hand side.
\end{proof}

\begin{lemma}\label{lA}
  The matrix $\Abbh$ defined in (\ref{eq:A_hom}) is
  symmetric and positive definite.
\end{lemma}
\begin{proof}
  From the definition we see that $\Abbh$ is symmetric and positive
  semidefinite. Therefore to show that it is positive definite we only
  have to see that it is injective.  Let $\xib \in \ZR^2$ be in the
  kernel of $\Abbh$. Thus $\Abbh\xib=0$ and thus
  \begin{equation*}
    \begin{aligned}
      0=\xib^T \Abbh \xib = \frac{1}{|\Gamma_Y|}\int_{\Gamma_Y} \xib^T \left(\tb +
        \partial_{\Gamma} \phib\right)\left( \tb +
        \partial_{\Gamma} \phib\right)^T \xib ds(\yb) =  \frac{1}{|\Gamma_Y|}
      \int_{\Gamma_Y} \left(\left( \tb + \partial_{\Gamma}
          \phib\right) \cdot \xib\right)^2 ds(\yb).
    \end{aligned}
  \end{equation*}
  This implies that
  \begin{equation*}
    \left(  \tb(\yb) + \partial_{\Gamma} \phib(\yb)\right) \cdot  \xib = 0, \qquad \mbox{ for a.e. } \yb \in \Gamma_Y.
  \end{equation*}

  Let $T \in \mathcal{T} (\Gamma_{Y}')$ be an arbitrary trail of $\Gamma_Y'$
  parametrized by $\gamb_T$, which connects opposite points $\yb_1$ and $\yb_1 + \eb_1$ that exist by \eqref{opposite}.
  By summing up results of Lemma~\ref{LemaIntegralFirstKind} for the functions
  $\yb \mapsto \phib(\yb)$ and $\yb \mapsto y_j$ for $j\in\{1,2\}$ we find
  \begin{equation*}
    0 =\int_{\Gamma_Y} (\partial_{\Gamma} \phib(\yb) + \tb (\yb))
    \chi_T(\yb) \,\mathrm{d}s(\yb) \cdot \xib =  \0b \cdot \xib
    + \eb_1 \cdot \xib = \eb_1 \cdot \xib.
  \end{equation*}

  Similarly, we take an arbitrary trail $T\in \mathcal{T} (\Gamma'_Y)$
  connecting opposite points $\yb_2$ and $\yb_2 + \eb_2$ that again exists by \eqref{opposite}
  to obtain $\eb_2 \cdot \xib=0$.
  Therefore $\eb_1\cdot \xib = \eb_2\cdot \xib = 0 $ which implies
  $\xib=0$ and so $\Abbh$ is injective.
\end{proof}

The following lemma states the existence and uniqueness of the limit problem~\eqref{P}.

\begin{lemma}[Well-posedness of the limit problem]\label{lmodel}
  The problem~\eqref{P} has a unique solution that satisfies
  \begin{equation*}
    u^0 \in \Ltwo(0,\tf;\cH)\cap \Ltwo(0,\tf; \Htwo(\Omega)), \qquad
    \partial_t{u}^0 \in \Ltwo(0,\tf;\Ltwo(\Omega)).
  \end{equation*}
\end{lemma}
\begin{proof}
   The statement is a direct consequence of the
   symmetry and the positive-definitness of the matrix $\Abbh$ by Lemma~\ref{lA}, the Poincar{\'e} inequality~\cite[Theorem~1 on page 558]{DL} and
   usual regularity theory for linear parabolic equations (see \cite{Evans}).
\end{proof}

\subsection{Proof of Theorem~\ref{theo:convergence}: mesh two-scale convergence to homogenized model}
\label{sec:homogenized_model:convergence}

\begin{proof}[Proof of Theorem~\ref{theo:convergence}]
  With the uniform stability estimates~\eqref{eq:stability}
  the assumptions of Theorem~\ref{tcomp1} are fulfilled, and there exist
  $u^0,z^0, w^0 \in \Ltwo(0,\tf;\Ltwo(\Omega;\Ltwo(\Gamma_Y)))$ and a sub-sequence of $(u^\delta)$ that is again denoted by $(u^\delta)$  such that
    \begin{align}
                     u^\delta &\xrightharpoonup{\mathrm{m2s}} u^0, &
     \partial_\Gamma u^\delta &\xrightharpoonup{\mathrm{m2s}} z^0, &
     \partial_t u^\delta &\xrightharpoonup{\mathrm{m2s}} w^0\ .
     \label{conv1}
  \end{align}
  The remainder of the proof is in seven steps.
  In step 1 and 2 we prove that $u^0$ is independent of $\yb$, first for strongly connected graphs
  and then for more general graphs.
  Then, in step 3 we show that
  $w^0 = \partial_t u^0$ and that there exists $u^1 \in \Ltwo(0,\tf;\Hoper)$
  with $\partial_t{u}^0 \in \Ltwo(0,\tf;\Ltwo(\Omega))$ such that
  $z^0 = \nabla u^0 \cdot \tb + \partial_{\Gamma,\yb} u^1$. We
  shall also prove that boundary (step 4) and initial conditions (step 5) are
  respected when passing to the limit.
  Finally, in step 6 and 7 we show that the limit solution $u^0$ satisfies
  the variational equation~\eqref{P:PDE} with the homogenized source $f_\homg$ and homogenized tensor $\Abbh$ defined by~\eqref{eq:f_hom} and \eqref{eq:A_hom}, respectively.

  \paragraph{1. Independence of $u^0$ from $\yb$ for strongly connected graphs.} %
  Let us suppose that the graph is strongly connected.
  We take an arbitrary
  $\ww \in C^\infty_c(\Omega)$
  and $\theta\in \Hper(\Gamma_Y)$, and an arbitrary
  {\cycle} $C \in \mathcal {C} (\Gamma_{Y,\#}') $. Since the function
  $\xb \mapsto
  u^\delta(\xb)\ww(\xb)\theta\left(\frac{\xb}{\delta}\right)$
  is in $\Hone_0(\Omegad)$ applying Lemma~\ref{LemaPhiIbyP} and multiplying the equations by $\delta^2\phi$, $\phi \in C^\infty([0,\tf])$ and integrating over $[0,\tf]$, we obtain
  \begin{align}
      \nonumber
      0 &= \delta^2 \int_0^{\tf} \phi(t) \int_{\Omegad} {\partial_\Gamma} \left( u^\delta (t,\xb) \ww
      \left(\xb \right)\theta\left(\frac{\xb}{\delta}\right) \right)
    \chi_C\left(\frac{\xb}{\delta}\right) \, \mathrm{d} s (\xb) \\
    \nonumber
	&=\delta^2 \int_0^{\tf} \phi(t) \int_{\Omegad} \bigg[
    \partial_\Gamma u^\delta(t,\xb) \ww
      \left(\xb \right)\theta\left(\frac{\xb}{\delta}\right) +
      u^\delta(t,\xb) \nabla \ww \left(\xb
      \right) \cdot \tb\left( \frac{\xb}{\delta} \right) \theta\left(\frac{\xb}{\delta}\right) \bigg]
    \chi_C\left(\frac{\xb}{\delta}\right) \, \mathrm{d} s (\xb)
    \, \mathrm{d}t \\
    &\quad + \delta \int_0^{\tf} \phi(t) \int_{\Omegad} \bigg[
      u^\delta(t,\xb) \ww \left(\xb
      \right)\partial_{\Gamma} \theta\left(\frac{\xb}{\delta}\right) \bigg]
    \chi_C\left(\frac{\xb}{\delta}\right) \, \mathrm{d} s (\xb) \, \mathrm{d}t.
    \label{eq:vd_with_w_and_phi}
  \end{align}
  Here, we used~\eqref{eq:deriv_on_mesh:micro_and_macro} and~\eqref{eq:deriv_on_mesh:macro}.
  Now, we let $\delta \to 0$ and apply the
  convergences in \eqref{conv1} on all three addends in
  the integral on the right hand side. Due to the factor~$\delta^2$,
  the first two addends vanish while the last remains leading to
  \begin{equation*}
    0 =  \int_0^{\tf} \phi(t) \int_{\Omega}\int_{\Gamma_Y} u^0(t, \xb,\yb) \ww (\xb)
    \partial_{\Gamma} \theta \left(\yb\right) \chi_C
    (\yb) \,\mathrm{d}s(\yb) \, \mathrm{d} \xb \, \mathrm{d}t.
  \end{equation*}
  Then, the arbitrariness of $\ww$ and $\phi$ implies
  \begin{equation*}
    0 =  \int_{\Gamma_Y} u^0(t, \xb,\yb) \partial_{\Gamma} \theta
    \left(\yb\right) \chi_C (\yb) \, \mathrm{d}s(\yb)
  \end{equation*}
  and in view of Lemma~\ref{LemaPhiConst} we find that $u^0$ is constant on $\Gamma_Y$.

\paragraph{2. Independence of $u^0$ from $\yb$ for general graphs}
Let $\Gamma_Y$ be an arbitrary graph in the unit cell that satisfies the
assumptions of the Theorem~\ref{theo:convergence}.
Let $\Gamma_{Y,-}$ be the graph obtained from $\Gamma_Y$ by
reversing orientations of all edges, and let
$\widetilde{\Gamma}_Y : = \Gamma_Y \cup \Gamma_{Y,-}$.  This graph has the same
number of vertices as $\Gamma_Y$ and the double number of edges which
topologically coincide, but we include both possible orientations.  Thus, for
each edge $e_0$ in $\Gamma_Y$ there are two edges $e_1$ and $e_2$ in $\widetilde{\Gamma}_Y$
which coincide with $e_0$, and such that $e_0$ and $e_1$ have the same
orientation, opposite from the orientation of $e_2$.  It is clear that $G_Y$ is
strongly connected.  We analogously define $\Gamma^\delta_-$ and
$\widetilde{\Gamma}^\delta$.

For all $\delta>0$ and $u^\delta$ on $[0,\tf] \times \Gamma^\delta$ we define
$\widetilde{u}^\delta$ on $[0,\tf] \times \widetilde{\Gamma}^\delta$ such that
the value on both doubled edges in $\Gamma^\delta$ coincide with the values of
$u^\delta$ on the original edge.  In other words, for original edge $e_0$ in
$\Gamma^\delta$ and its copies $e_1, e_2$ in $\widetilde{\Gamma}^\delta$ (one
oriented equally, one oriented oppositely) we have
$$\widetilde{u}^\delta\big|_{\xb \in e_1} = \widetilde{u}^\delta\big|_{\xb \in e_2} = u^\delta\big|_{\xb \in e_0}.$$

We easily see that we also have
\begin{align*}
  \partial_t \widetilde{u}^\delta\big|_{\xb \in e_1} = \partial_t \widetilde{u}^\delta\big|_{\xb \in e_2} = \partial_t  u^\delta\big|_{\xb \in e_0}
  \text{ and } \partial_{\Gamma}{\widetilde{u}^\delta}\big|_{\xb \in e_1} = -\partial_{\Gamma}{\widetilde{u}^\delta}\big|_{\xb \in e_2} = \partial_{\Gamma} u^\delta\big|_{\xb \in e_0},
\end{align*}
so we see that families
\begin{equation*}
      \sqrt{\delta} \|\widetilde{u}^\delta\|_{\Ltwo(0,\tf;\Ltwo(\widetilde{\Gamma}^\delta))},
      \sqrt{\delta} \|\partial_{\Gamma} \widetilde{u}^\delta
    \|_{\Ltwo(0,\tf;\Ltwo(\widetilde{\Gamma}^\delta))} ,
     \sqrt{\delta}
      \|\partial_t{\nu}^\delta\|_{\Ltwo(0,\tf;\Ltwo(\widetilde{\Gamma}^\delta))}
  \end{equation*}
  are bounded (by  $2C$, where $C$ is the constant in~\eqref{eq:stability}), so all families have convergent
  subsequences (still denoted by $\delta$). Furthermore $\widetilde{\Gamma}_Y$
  is strongly connected, so we can apply the result obtained in the step 1.
  Thus the two-scale limit $\widetilde{u}^0$ of the sequence
  $(\widetilde{u}^\delta)_\delta$ is independent of the fast variable.  This
  means especially that $\widetilde{u}^0$ takes the same value for both edges
  in $\widetilde{\Gamma}_Y$ of one edge in $\Gamma_Y$.

  Now, we define
  $u^0: (0,\tf) \times \Omega \times \Gamma_Y \to \ZR, (t,\xb,\yb) \mapsto
  \widetilde{u}^0(t,\xb)$
  and show that it is the two-scale limit of $u^\delta$. For this we consider
  for each test function $\psi \in \Ltwo(0,\tf;\Ltwo(\Gamma_Y;C(\Omega)))$ a
  corresponding function
  $\widetilde{\psi} \in \Ltwo(0,\tf;\Ltwo(\widetilde{\Gamma}_Y;C(\Omega)))$ with
  $\widetilde{\psi}(\cdot,\cdot,\yb) := \psi(\cdot,\cdot,\yb)$ if
  $\yb \in \Gamma_Y$ and $\widetilde{\psi}(\cdot,\cdot,\yb) := 0$ otherwise.
  This yields
 \begin{equation*}
    \begin{aligned}
      &\hspace{-3ex} \delta \int_0^{\tf} \int_{\Omegad} u^\delta (t,\xb) \psi
      \left(t,\xb, \frac{\xb}{\delta}\right) \,\mathrm{d}s(\xb) \, \mathrm{d}t
      = \delta \int_0^{\tf} \int_{\widetilde{\Gamma}^\delta}
      \widetilde{u}^\delta (t,\xb) \widetilde{\psi} \left(t,\xb,
        \frac{\xb}{\delta}\right) \,\mathrm{d}s(\xb) \, \mathrm{d}t \\
      & \to \int_0^{\tf} \int_{\Omega} \int_{\widetilde{\Gamma}_Y}
      \widetilde{u}^0(t,\xb) \widetilde{\psi} (t,\xb, \yb)
      \,\mathrm{d}s(\yb)\,\mathrm{d}\xb \, \mathrm{d}t = \int_0^{\tf}
      \int_{\Omega} \int_{\Gamma_Y} u^0(t,\xb) \psi (t,\xb, \yb)
      \,\mathrm{d}s(\yb)\,\mathrm{d}\xb \, \mathrm{d}t.
     \end{aligned}
  \end{equation*}
This proves that $u^0$, which does not dependent on $\yb$ by its definition, is
the two scale limit of the sequence $(u^\delta)_\delta$.

  \paragraph{3. Form of limits.}  First we shall prove that
    $w^0 = \partial_t u^0$. To do so, we take an arbitrary admissible
  function $\psi$ such that $\psi(0,\cdot,\cdot) = \psi(\tf,\cdot,\cdot) = 0$, so we
  have
  \begin{equation*}
    0 = \delta \int_0^{\tf} \frac{\text{d}}{\text{d}t} \int_{\Omegad}
    u^\delta(t,\xb) \psi\Big( t,\xb,\frac{\xb}{\delta} \Big) \,
    \mathrm{d}s(\xb) \, \mathrm{d}t,
  \end{equation*}
  so that
  \begin{equation*}
    \delta \int_0^{\tf} \int_{\Omegad}
    \partial_t{u}^\delta(t,\xb) \psi\Big( t,\xb,\frac{\xb}{\delta} \Big) \,
    \mathrm{d}s(\xb) \, \mathrm{d}t = - \delta \int_0^{\tf}
    \int_{\Omegad}
    u^\delta(t,\xb) \partial_t \psi\Big(
    t,\xb,\frac{\xb}{\delta} \Big) \,
    \mathrm{d}s(\xb) \, \mathrm{d}t.
  \end{equation*}
  Using then the two-scale convergence of $u^\delta$ to $u^0$ and
  $\partial_t{u}^\delta$ to $w^0$ leads to
  \begin{equation*}
    \int_0^{\tf}\int_{\Omega} \int_{\Gamma_Y} w^0(t,\xb,\yb) \psi (t,\xb, \yb)
    \,\mathrm{d}s(\yb)\,\mathrm{d}\xb \, \mathrm{d}t = -
    \int_0^{\tf}\int_{\Omega} \int_{\Gamma_Y} u^0(t,\xb)
    \partial_t \psi (t,\xb, \yb)
    \,\mathrm{d}s(\yb)\,\mathrm{d}\xb \, \mathrm{d}t.
  \end{equation*}
  Due to the arbitrariness of the function $\psi$, it turns out that
  $\partial_t{u}^0$ exists, that $\partial_t{u}^0 = w^0$ and thus
  $\partial_t{u}^0 \in \Ltwo(0,\tf;\Ltwo(\Omega))$.

  Now, we shall prove that for the limit $z^0 =\nabla u^0\cdot \tb + \partial_{\Gamma,\yb} u^1$ holds. To do
  so, let us consider an arbitrary {\cycle} $C \in \mathcal {C} (\Gamma_{Y,\#}')$ and an arbitrary
  $\ww \in \cD (\Omega)$. Since the function
  $\xb \mapsto v^\delta (t,\xb) \ww \left(\xb \right)$
  is in $H^1_0(\Gamma^\delta)$ applying Lemma~\ref{LemaPhiIbyP} and multiplying the equality by $\delta\phi$ with a function
  $\phi\in C^\infty([0,\tf])$ and integrating over $[0,\tf]$ leads to
  $$
  \aligned
  0 &= \delta \int_0^{\tf} \phi(t) \int_{\Omegad} \partial_\Gamma \left( u^\delta (t,\xb) \ww \left(\xb
      \right) \right) \chi_C\left(\frac{\xb}{\delta}\right)\,\mathrm{d} s (\xb)\ .
  \endaligned
  $$
  Then~\eqref{eq:deriv_on_mesh:macro} implies
  $$
  \aligned
     &\delta \int_0^{\tf} \phi(t)\int_{\Omegad} \partial_\Gamma u^\delta (t,\xb) \ww(\xb)
    \chi_C\left(\frac{\xb}{\delta}\right) \, \mathrm{d} s(\xb) \,
    \mathrm{d}t \\
    &\qquad \hfill =
    - \delta \int_0^{\tf} \phi(t) \int_{\Omegad} u^\delta(t,\xb) \nabla
    \ww (\xb) \cdot \tb\left(\frac{\xb}{\delta}\right)\chi_C\left(\frac{\xb}{\delta}\right) \,
    \mathrm{d} s (\xb) \, \mathrm{d}t.
    \endaligned
  $$
  Now, taking the limit $\delta\to0$ and applying the mesh two scale limits~\eqref{conv1} on each side of the equation we
  obtain that
  \begin{multline*}
    \int_0^{\tf} \phi(t) \int_{\Omega} \int_{\Gamma_Y} z^0(t,\xb,\yb) \ww(\xb)
    \chi_C(\yb) \, \mathrm{d} s(\yb) \, \mathrm{d} \xb  \,
    \mathrm{d}t \\ = -
    \int_0^{\tf} \phi(t) \int_{\Omega} \int_{\Gamma_Y} u^0(t,\xb) \nabla \ww (\xb) \cdot \tb(\yb)
    \chi_C(\yb) \, \mathrm{d} s(\yb) \, \mathrm{d} \xb \, \mathrm{d}t,
  \end{multline*}
  for all $\ww \in C^\infty_c(\Omega)$ and all
  $\phi \in C^\infty([0,T])$.  Arbitrariness of $\phi$ then implies
    \begin{equation*}
      \int_{\Omega} \int_{\Gamma_Y} z^0(t,\xb,\yb) \ww(\xb)
      \chi_C(\yb) \, \mathrm{d} s(\yb) \, \mathrm{d} \xb = -
      \int_{\Omega}  u^0(t,\xb) \nabla_{\xb} \ww (\xb) \cdot \left( \int_{\Gamma_Y}\tb(\yb)
        \chi_C(\yb) \, \mathrm{d} s(\yb) \right)\, \mathrm{d} \xb .
  \end{equation*}
  Let us denote
  \begin{equation*}
    \tb_C := \int_{\Gamma_Y}\tb(\yb) \chi_C(\yb) \, \mathrm{d}s (\yb).
  \end{equation*}
Thus we obtain
  \begin{equation}
    \label{eq:form_of_limits:1}
    \int_{\Omega} \left(\int_{\Gamma_Y} z^0(t,\xb,\yb)  \chi_C(\yb) \, \mathrm{d} s(\yb) \right)\ww(\xb)
    \, \mathrm{d} \xb = -
    \int_{\Omega}  u^0(t,\xb) \nabla \ww (\xb) \cdot \tb_C\, \mathrm{d} \xb .
  \end{equation}
  Then, by assumption on $\Gamma_Y$, if we consider trails $T_i$ that
  connect the opposite points $\yb_i$ and $\yb_i + \eb_i$ (\ie, the first and the last vertices
  are $\yb_i$ and $\yb_i + \eb_i$) $i=1,2$ it is a {\cycle} $C_i$ on $\Gamma_{Y,\#}'$.
  Now,  using Lemma~\ref{LemaIntegralFirstKind}
  with the function $\varphi_j(\yb) = \yb \cdot \eb_j$ (then~\eqref{eq:deriv_on_mesh:macro} implies
  $\partial_{\Gamma} \varphi_j(\yb) = \tb(\yb) \cdot \eb_j$) for
  $j \in \lbrace 1, 2\rbrace$ leads to
  \begin{align*}
    \tb_{C_i}\cdot\eb_j &:= \int_{\Gamma_Y}\tb(\yb)\cdot\eb_j \chi_{C_i}(\yb) \, \mathrm{d}s (\yb)
    = \int_{\Gamma_Y}\tb(\yb)\cdot\eb_j \chi_{T_i}(\yb) \, \mathrm{d}s (\yb)
    = \int_{\Gamma_Y}\partial_{\Gamma} \varphi_j(\yb) \chi_{T_i}(\yb) \, \mathrm{d}s (\yb)\\
    &= \varphi_j(\yb_i+\eb_i) - \varphi_j(\yb_i) = \eb_i\cdot\eb_j
  \end{align*}
  and, hence,  $\tb_{C_i} = \eb_i$, $i\in \{1,2\}$.
  That means that there are at least two linear independent vectors~$\tb_C$
  and the equality~\eqref{eq:form_of_limits:1} defines the weak derivative $\nabla_\xb u^0$.
%
  Since the function
  \begin{equation*}
    (t,\xb) \mapsto \int_{\Gamma_Y} z^0(t,\xb,\yb)  \chi_C(\yb) \, \mathrm{d} s(\yb)
  \end{equation*}
  belongs to $\Ltwo(0,\tf;\Ltwo(\Omega))$  the weak
  derivative $\nabla_{\xb} u^0$ belongs to
  $\Ltwo(0,\tf;\Ltwo(\Omega))$ as well. Hence, $u^0 \in \Ltwo(0,\tf;\Hone(\Omega))$.
  Now, with the equality $\nabla w\cdot\tb_C = \div(w \tb_C)$ as $\tb_C$ is a constant vector
  integrating by parts on the right
  hand side of~\eqref{eq:form_of_limits:1} we obtain
  \begin{equation*}
    \int_{\Omega} \int_{\Gamma_Y} \ww(\xb) \left( z^0(t,\xb,\yb)
      - \nabla_\xb u^0 (t, \xb) \cdot \tb(\yb) \right) \chi_C(\yb)
    \,\mathrm{d}s(\yb) \,\mathrm{d}\xb= 0.
  \end{equation*}
  Then, arbitrariness of $\ww$ implies
  \begin{equation*}
    \int_{\Gamma_Y} \left( z^0(t,\xb,\yb)  - \nabla_\xb u^0 (t, \xb) \cdot \tb(\yb) \right) \chi_C(\yb) \,\mathrm{d} s(\yb)= 0.
  \end{equation*}
  As we assumed arbitrariness of $C \in \cC(\Gamma_{Y,\#}')$
  Lemma~\ref{RemarkAntider} now implies that there exists a function $u^1(t,\xb,\cdot) \in \Hper(\Gamma_Y)$ such that
  \begin{equation}
    \label{eq:form_of_limits:2}
    \partial_{\Gamma,\yb} u^1(t, \xb,\yb) =
    z^0(t, \xb,\yb) - \nabla_\xb u^0 (t, \xb) \cdot \tb(\yb).
  \end{equation}
  Since the right hand side is in
  $\Ltwo(0,\tf;\Ltwo(\Omega;\Ltwo(\Gamma_Y)))$ we conclude that $u^1 \in
  \Ltwo(0,\tf;\Hper(\Gamma_Y))$.

  \paragraph{4. Boundary conditions.}  Let us take a function $\ww\in C^1(\Omega)$
  such that $\ww|_{\partial\Omega \backslash \Gamma_D} =0$, \ie~it vanishes on
  the complement of the boundary of the Dirichlet boundary.
  As the function
  $\xb \mapsto \delta u^\delta(t, \xb)\ww(\xb)$ is in $H^1_0(\Omegad)$ it follows from
  Lemma~\ref{LemaPhiIbyP} for circuits introduced in the previous step that
  \begin{equation*}
    \delta\int_{\Omegad} \partial_{\Gamma} u^\delta(t, \xb) \ww(\xb) \chi_{C_i}\left(\frac{\xb}{\delta}\right)
    \,\mathrm{d}s(\xb) = - \delta \int_{\Omegad} u^\delta(t, \xb)
    \partial_\Gamma \ww(\xb) \chi_{C_i}\left(\frac{\xb}{\delta}\right)
    \,\mathrm{d}s(\xb), \quad i=1,2.
  \end{equation*}
  Multiplying the equation by $\phi\in C^\infty([0,\tf])$ and integrating
  over $[0,\tf]$ we obtain in the limit $\delta\to0$ using~\eqref{eq:form_of_limits:2} and~\eqref{eq:deriv_on_mesh:macro}
  \begin{equation}
  \aligned
    &\int_0^{\tf} \phi(t)\int_{\Omega}\int_{\Gamma_Y} (\nabla_\xb u^0 (t,\xb)
    \cdot \tb(\yb)
    + \partial_{\Gamma,\yb} u^1 (t,\xb,\yb))\ww(\xb) \chi_{C_i}(\yb)
    \,\mathrm{d}s(\yb) \,\mathrm{d}\xb \, \mathrm{d}t\\
    &\qquad = -
    \int_0^{\tf} \phi(t) \int_{\Omega}\int_{\Gamma_Y} u^0(t,\xb)
    \nabla\ww(\xb) \cdot \tb(\yb) \chi_{C_i}(\yb) \,\mathrm{d}s(\yb) \,\mathrm{d}\xb\, \mathrm{d}t.
    \endaligned
    \label{eq:integral_v0_v1_particular_psi}
  \end{equation}
  Therefore, using $\tb_{C_i}=\eb_i$, $i=1,2$, \eqref{eq:integral_v0_v1_particular_psi}
  becomes
  $$
  \aligned
    &\int_0^{\tf} \phi(t) \int_{\Omega} \left(\nabla u^0(t,\xb) \cdot \eb_i +
      \int_{\Gamma_Y}\partial_{\Gamma,\yb} u^1 (t,\xb,\yb) \chi_{C_i}(\yb)
      \,\mathrm{d} s(\yb)\right)\ww(\xb) \,\mathrm{d} \xb \,
    \mathrm{d}t \\
    &\qquad = -
     \int_0^{\tf} \phi(t)\int_{\Omega} u^0(t,\xb) \nabla \ww(\xb)\cdot \eb_i \,\mathrm{d}\xb \, \mathrm{d}t, \quad i=1,2.
     \endaligned
  $$
  Since $u^1$ is $\Gamma_Y$ periodic, using
    Lemma~\ref{LemaIntegralFirstKind} with $\cG = \Gamma_Y'$, it holds
  \begin{equation*}
    \int_{\Gamma_Y}\partial_{\Gamma,\yb} u^1 (\xb,\yb) \chi_{C_i}(\yb) \,\mathrm{d}s(\yb) = 0, \quad i=1,2.
  \end{equation*}
  Thus, using arbitrariness of $\phi\in C^\infty([0,T])$ we obtain that
  \begin{equation*}
    \eb_i \cdot \int_{\Omega} \nabla u^0 (t,\xb) \ww(\xb)
    \,\mathrm{d}\xb = - \eb_i \cdot \int_{\Omega} u^0(t,\xb) \nabla
    \ww(\xb) \,\mathrm{d}\xb.
  \end{equation*}
  After partial integration in the right hand side it follows that
  \begin{equation*}
    \eb_i \cdot  \int_{\Gamma_D} u^0(t,\xb) \ww(\xb) \nb(\xb)
    \,\mathrm{d}s(\xb) = 0,
  \end{equation*}
  where $\nb$ is the unit outer normal on $\partial \Omega$. Due to the
  arbitrariness of the function $\ww$, we deduce that the function
  $(t,x) \mapsto \eb_i \cdot \nb(\xb) u^0(t,\xb)$ vanishes on
  $(0,\tf) \times \Gamma_D$.
  This is only possible for $i=1,2$ if $u^0$ vanishes on $\Gamma_D$.

  \paragraph{5. Initial condition.} Let us take that
  $u^\delta\big|_{t=0}$ mesh two-scale converges to the limit
  $u_{\init} \in \Ltwo(0,\tf;\Loper)$. Let us take any
  $\psi \in C^\infty_0(\Omega, \Hper(\Gamma_Y)) $ and
  $\phi \in C^{\infty}(0,\tf)$ such that $\phi(0)=1, \phi(\tf)=0$. Then
  partial integration in $t$ gives us
  \begin{multline*}
      \delta \int_{0}^{\tf} \int_{\Omegad} \partial_t{u}^\delta \left(t,\xb\right) \psi\left(\xb,\frac{\xb}{\delta}\right) \phi(t) \,\mathrm{d}s(\xb) \,\mathrm{d} t \\
      = - \delta \int_{0}^{\tf} \int_{\Omegad} u^\delta \left(t,\xb\right) \psi\left(\xb,\frac{\xb}{\delta}\right) \phi'(t) \,\mathrm{d}s(\xb) \mathrm d t
      - \delta \int_{\Omegad} u^\delta \left(0,\xb\right) \psi\left(\xb,\frac{\xb}{\delta}\right)
      \,\mathrm{d}s(\xb) .
  \end{multline*}
  Now, in the limit $\delta\to 0$ we obtain
  \begin{multline*}
      \int_{0}^{\tf} \int_{\Omega} \int_{\Gamma_Y} \partial_t u^0 \left(t,\xb\right) \psi\left(\xb,\yb\right) \phi(t) \,\mathrm{d}s(\yb) \mathrm d \xb \mathrm d t  \\
      = - \int_{0}^{\tf} \int_{\Omega}\int_{\Gamma_Y} u^0 \left(t,\xb\right) \psi\left(\xb,\yb\right) \phi'(t) \,\mathrm{d}s(\yb)\mathrm d \xb \mathrm d t
      - \int_{\Omega}\int_{\Gamma_Y} u_{\init} \left(\xb\right) \psi\left(\xb,\yb\right)
      \, \mathrm{d}s(\yb)\mathrm d \xb .
  \end{multline*}
  Integrating the left hand side by parts we find
  \begin{multline*}
      \int_{0}^{\tf} \int_{\Omega} \int_{\Gamma_Y} \partial_t u^0 \left(t,\xb\right) \psi\left(\xb,\yb\right) \phi(t) \,\mathrm{d}s(\yb) \mathrm d \xb \mathrm d t = \\
      - \int_{0}^{\tf} \int_{\Omega}\int_{\Gamma_Y} u^0 \left(t,\xb\right) \psi\left(\xb,\yb\right) \phi'(t) \,\mathrm{d}s(\yb)\mathrm d \xb \mathrm d t
      - \int_{\Omega}\int_{\Gamma_Y} u^0 \left(0,\xb\right) \psi\left(\xb,\yb\right)
      \, \mathrm{d}s(\yb)\mathrm d \xb .
  \end{multline*}
  By substracting last two results we obtain
  \begin{equation*}
      \int_{\Omega}\int_{\Gamma_Y} \left( u_{\init} (\xb)- u^0 \left(0,\xb\right) \right) \psi\left(\xb,\yb\right)
      \, \mathrm{d}s(\yb)\mathrm d \xb = 0.
  \end{equation*}
  Arbitrarines of $\psi$ implies $u_\textbf{\init} = u^0\big|_{t=0}$.


 \paragraph{6. Form of the corrector term $u^1$.}
 We consider a test function $v^\delta(\xb) = v(\xb)\theta(\frac{\xb}{\delta})$ as product of a slow varying function $v \in \cH$
 and a fast varying function $\theta \in \Hper(\Gamma_Y)$.
Multiplying~\eqref{deltaP:1} by $\delta^2 \phi$ for $\phi\in C^\infty([0,\tf])$,  integrating over $[0,\tf]$ and using this particular test function leads to
\begin{multline}\label{deltaP1}
  \delta^2 \int_0^{\tf} \phi(t) \int_{\Omegad} \partial_t{u}^\delta(t,\xb)
  \theta\left(\frac{\xb}{\delta}\right) v(\xb)
  \,\mathrm{d}s(\xb)  \, \mathrm{d}t \\
  + \delta \int_0^{\tf} \phi (t) \int_{\Omegad} a(\xb) \partial_{\Gamma}
  u^\delta(t,\xb) \left( \partial_{\Gamma,\yb}
    \theta\left(\frac{\xb}{\delta}\right) v(\xb)+ \delta
    \theta\left(\frac{\xb}{\delta}\right) \nabla v
    (\xb) \cdot \tb \left(\frac{\xb}{\delta}\right) \right) \,\mathrm{d}s(\xb) \, \mathrm{d}t\\ = \delta^2
  \int_0^{\tf} \phi(t)\int_{\Omegad} f\left(t,\xb,\frac{\xb}{\delta}\right)
  \theta\left(\frac{\xb}{\delta}\right) v(\xb) \,\mathrm{d}s(\xb)
  \, \mathrm{d}t.
\end{multline}
Under the hypothesis of Lemma~\ref{lema:bound_data} and using the
Cauchy-Schwartz inequality, the right-hand side of~\eqref{deltaP1}
tends to $0$ as $\delta$ tends to $0$, independently of the choice of
$v$, $\theta$ and $\phi$.  Furthermore the first term on the left hand
side also tends to zero by the \emph{a priori} estimates from
Lemma~\ref{lapriori}.  The second term in the left hand side of
(\ref{deltaP1}) is split in two.  We use the Cauchy-Schwarz inequality
and the \emph{a priori} estimate from Lemma~\ref{lapriori} for
$\partial_{\Gamma} u^\delta$ to conclude that the term with additional
$\delta$ tends to zero as $\delta \to 0$.  We use then the the mesh two-scale
convergence of $\partial_{\Gamma} u^\delta$ to $\partial_{\Gamma} u^0
+ \partial_{\Gamma,\yb} u^1$ for the test function
\begin{equation*}
  \psi:(\xb,\yb) \mapsto a(\xb) v(\xb) \partial_{\Gamma} \theta(\yb)
  \in \Ltwo(\Omega ; \Ltwo(\Gamma_Y) ),
\end{equation*}
to take the limit in the remaining term on the left hand side of
\eqref{deltaP1}. What remains in the limit is
\begin{equation}
  \label{deltaP2}
  \int_0^{\tf} \phi(t) \int_{\Omega} \int_{\Gamma_Y} a(\xb) \left(
    \partial_{\Gamma} u^0(t,\xb) + \partial_{\Gamma,\yb}
    u^1(t,\xb,\yb) \right) v(\xb) \partial_{\Gamma}
  \theta(\yb) \,\mathrm{d}s(\yb)
  \, \mathrm{d}\xb\, \mathrm{d}t =0.
\end{equation}
Since $v$ is an arbitrary function of $\cH$ and $\phi$ arbitrary in $C^\infty([0,\tf])$, we obtain that
\begin{equation*}
  \int_{\Gamma_Y} \left( \partial_{\Gamma} u^0(t,\xb)
    + \partial_{\Gamma,\yb} u^1(t,\xb,\yb) \right) \partial_{\Gamma}
  \theta(\yb) \,\mathrm{d}s(\yb) =0.
\end{equation*}
Using~\eqref{eq:deriv_on_mesh:macro} this gives the unit pattern problem for $u^1$:
\begin{equation}\label{deltaP3}
  \int_{\Gamma_Y} \partial_{\Gamma,\yb} u^1(t,\xb, \yb)
  \partial_{\Gamma} \psi(\yb) \,\mathrm{d}s(\yb) = -
  \nabla_{\xb} u^0(t,\xb) \cdot \int_{\Gamma_Y}  \tb(\xb)
  \partial_{\Gamma} \psi(\yb) \,\mathrm{d}s(\yb).
\end{equation}
which has by Lemma~\ref{lcanon} a unique solution up to an additive constant in $\yb$ (note that $t$ and $\xb$ are parameters here).
With this, we can express the solution of \eqref{deltaP3}
using the solution of the canonical problems~\eqref{canon} as
$u^1 (t,\xb, \yb)
  = \nabla_\xb u^0 (t,\xb) \cdot
  \phib(\yb) + \tilde{u}^1(t,\xb) $
with a function $\tilde{u}^1$
independent of the fast variable $\yb$ and
\begin{equation}\label{u1}
  \partial_{\Gamma,\yb} u^1(t,\xb, \yb) = \nabla_\xb u^0(t,\xb) \cdot \nabla_{\Gamma} \phib(\yb).
\end{equation}

 \paragraph{7. Variational formulation for the unique limit solution $u^0$.}

 Next we take a test function
$v^\delta(t,\xb) = \delta v(\xb)\phi(t)$, where $v \in \Hone(\Omega)$ is
independent of the fast variable $\yb$ and $\phi \in C^\infty([0,\tf])$,
in (\ref{deltaP}). We obtain
\begin{equation*}
  \aligned
  &\delta \int_0^{\tf} \phi(t) \int_{\Omegad} \partial_t{u}^\delta(t,\xb) v(\xb)
  \,\mathrm{d}s(\xb)  \, \mathrm{d}t
  + \delta \int_0^{\tf} \phi(t)\int_{\Omegad} a(\xb) \partial_{\Gamma} u^\delta(\xb) \partial_{\Gamma} v(\xb)
  \,\mathrm{d} s(\xb) \,\mathrm{d}t \\
  &\qquad = \delta \int_0^{\tf} \phi(t) \int_{\Omegad} f(t,\xb) v(\xb)
  \,\mathrm{d} s(\xb) \,\mathrm{d}t.
  \endaligned
\end{equation*}
Then using the mesh two-scale convergence of $\partial_{\Gamma} u^\delta$
  (respectively $\partial_t u^\delta$) to
  $\partial_{\Gamma} u^0 + \partial_{\Gamma,\yb} u^1$ (resp.
  $\partial_t u^0$), we take the limit when $\delta$ tends to zero and obtain
\begin{equation*}
\aligned
&\int_0^{\tf} \phi(t) \int_{\Omega} \int_{\Gamma_Y}\partial_t{u}^0(t,\xb) v(\xb) \, \mathrm{d}s(\yb)
  \,\mathrm{d}s(\xb)  \, \mathrm{d}t \\
&\qquad\quad   +  \int_0^{\tf} \phi(t) \int_{\Omega} \int_{\Gamma_Y} a(\xb) \left( \nabla
    u^0(\xb) \cdot \tb(\yb) + \partial_{\Gamma,\yb} u^1(\xb, \yb)\right) \nabla
  v(\xb) \cdot \tb(\yb)
  (\xb) \,\mathrm{d} s(\yb) \,\mathrm{d} \xb \, \mathrm{d}t\\
&\qquad   = |\Gamma_Y| \int_0^{\tf} \phi(t)\int_\Omega
  f_\homg(t,\xb) v(\xb)  \,\mathrm{d} \xb\, \mathrm{d}t,
  \endaligned
\end{equation*}
where $f_\homg$ was defined in~\eqref{eq:f_hom}.
Using arbitrariness of $\phi$ in $C^\infty([0,\tf])$ we
obtain that for almost every $t\in [0,\tf]$ we have
\begin{equation*}
  \aligned
  &
  \int_{\Omega}  \partial_t{u}^0(t,\xb) v(\xb) \,\mathrm{d}s(\xb)
  + \frac{1}{|\Gamma_Y|} \int_{\Omega} \int_{\Gamma_Y} a(\xb) \left( \nabla
    u^0(t,\xb) \cdot \tb(\yb) + \partial_{\Gamma,\yb} u^1(t,\xb,
    \yb)\right) \nabla v
  (\xb) \cdot \tb(\yb)\,\mathrm{d} s(\yb) \,\mathrm{d} \xb \\
  &\qquad = \int_\Omega
  f_\homg(t,\xb) v(\xb)  \,\mathrm{d} \xb,
  \endaligned
\end{equation*}
Using the solution representation for $\partial_{\Gamma,\yb} u^1$
from~\eqref{u1},
$\partial_{\Gamma,\yb} u^1(t,\xb, \yb) = \nabla u^0(t,\xb) \cdot
\nabla_{\Gamma} \phib(\yb)$, we obtain
\begin{equation}\label{eq:last}
  \aligned
 &  \int_{\Omega}  \partial_t{u}^0(t,\xb) v(\xb) \,\mathrm{d}s(\xb)
  + \frac{1}{|\Gamma_Y|} \int_\Omega a(\xb) \bigg( \int_{\Gamma_Y}
  \tb(\yb) \left(  \tb(\yb) + \partial_{\Gamma} \phib(\yb)\right)^T
  \,\mathrm{d} s(\yb) \bigg)
  \nabla u^0(t,\xb) \cdot \nabla v(\xb) \, \mathrm{d} \xb\\
 &\qquad  =
  \int_\Omega f_\homg(t,\xb) v(\xb) \,\mathrm{d}\xb.
  \endaligned
\end{equation}
Testing the canonical problem \eqref{canon} by the two functions $\psi=\phi_i$, $i=1,2$ we obtain that
  \begin{equation*}
    \int_{\Gamma_Y} \partial_{\Gamma} \phib(\yb)
    (\tb(\yb) + \partial_{\Gamma} \phib(\yb))^T \mathrm{d}s(\yb)=0.
  \end{equation*}
Thus the matrix  $\Abbh$, defined in~\eqref{eq:A_hom}, can be also written as
$$
\Abbh = \int_{\Gamma_Y}
  \tb(\yb) \left(  \tb(\yb) + \partial_{\Gamma} \phib(\yb)\right)^T  \,\mathrm{d} s(\yb)
$$
and, hence, in view of~\eqref{eq:last} the two-scale limit $u^0$ satisfies~\eqref{P:PDE}.
Since $\Abbh$ is positive definite by Lemma~\ref{lA}
similarly to Lemma~\ref{ldeltaexist}
existence and uniqueness of the limit problem \eqref{P} follow.

As usual, uniqueness of the solution of the
limit problem then also implies the mesh two-scale convergences for the whole
family $(u^\delta)_\delta$ and their derivatives in \eqref{conv1}.
\end{proof}

In Theorem~\ref{theo:convergence} we assumed the stability estimate of the solutions $u^\delta$ of~\eqref{eq:Heat_equation} stated in Theorem~\ref{theo:wellposedness}. Repeating the first five steps in the proof of Theorem~\ref{theo:convergence} for any sequence of function $(v^\delta)_\delta$ satisfying the same stability estimate we can state similarly the following statement.
\begin{corollary}
  Let $\Gamma_Y$ be connected and satisfies \eqref{opposite}.
  Let the sequence $(v^\delta)_\delta \in \Ltwo(0,\tf;\Hone(\Omegad))$ with $\delta\to0$ be such that
  $\partial_t{v}^\delta \in \Ltwo(0,\tf;\Ltwo(\Omegad))$ for any $\delta > 0$ be
  such that there is $C>0$ such that for all $\delta>0$ one has
  \begin{equation*}
      \sqrt{\delta} \|v^\delta\|_{\Ltwo(0,\tf;\Ltwo(\Omegad))} \leqslant C,\qquad
      \sqrt{\delta} \|\partial v^\delta\|_{\Ltwo(0,\tf;\Ltwo(\Omegad))} \leqslant C,\qquad
      \sqrt{\delta}
      \|\partial_t{v}^\delta\|_{\Ltwo(0,\tf;\Ltwo(\Omegad))}  \leqslant
      C.
  \end{equation*}
  Then there is a sub-sequence of $(v^\delta)_\delta$ (still denoted
  by $v^\delta$) and functions $v^0 \in \Ltwo(0,\tf;\Hone(\Omega))$ and
  $v^1 \in \Ltwo(0,\tf;\Hper(\Gamma_Y))$ with
  $\partial_t{v}^0 \in \Ltwo(0,\tf;\Ltwo(\Omega))$ such that
  \begin{equation*}
                     v^\delta \xrightharpoonup{\mathrm{m2s}} v^0, \qquad
     \partial_\Gamma v^\delta \xrightharpoonup{\mathrm{m2s}} \partial_{\Gamma,\yb} v^1 + \nabla v^0 \cdot \tb,  \qquad
     \partial_t v^\delta \xrightharpoonup{\mathrm{m2s}} \partial_t v^0\ .
  \end{equation*}
  Additionally, if $(v^\delta)_\delta \in \Ltwo(0,\tf;\cH^\delta)$, then the limit $v^0$
  belongs to $\Ltwo(0,\tf;\cH)$. Also, if $v^\delta\big|_{t=0}$ mesh two-scale converges, its limit is $v^0\big|_{t=0}$.
\end{corollary}

\subsection{Proof of Theorem~\ref{theo:convergence_in_norm}: convergence in norm to homogenized model}
\label{sec:homogenized_model:convergence_in_norm}


\begin{proof}[Proof of Theorem~\ref{theo:convergence_in_norm}]
By assumption on  $f$ we find that  $f_\homg \in \Ltwo(0,\tf;\Htwo(\Omega))$ and $\partial_t f_\homg \in \Ltwo(0,\tf;\Ltwo(\Omega))$
and by assumption on $ u_\init$ using~\cite[Chap. 7, Theorem 6]{Evans} we conclude
\begin{equation*}
  u^0 \in L^2(0,\tf;\Hsp{4}(\Omega)), \qquad \partial_t u^0 \in L^2(0,\tf;\Htwo(\Omega)).
\end{equation*}
Therefore $ \nabla u^0 \in L^2(0,\tf;H^3(\Omega))$ and thus
\begin{equation*}
  (t,\xb,\yb) \mapsto \nabla_\xb
      u^1(t,\xb,\yb) \cdot \tb(\yb) \in
  \Ltwo(0,\tf;\Htwo(\Omega;\Ltwo(\Gamma_Y))), \qquad
  \partial_{\Gamma,\yb}u^1 \in L^2(0,\tf;\Hsp{3}(\Omega;L^2(\Gamma_Y))).
\end{equation*}
Moreover $f\in \Ltwo(0,\tf;\Ltwo (\Gamma_Y;C(\Omega)))$, so $f$ is an admissible test function.

Therefore for any $\delta$ the integrals in the following $\delta$ family
and all the terms in the following computation are well defined
\begin{equation}
  \label{eq:Lambda_delta}
  \begin{aligned}
    \Lambda^\delta (t) : =& \frac{1}{2}\delta \int_{\Omegad}
    \left( u^\delta(t,\xb) - u^0(t,\xb)\right)^2 \,
    \mathrm{d}s(\xb)\\
    &+ \delta \int_0^t \int_{\Omegad} a(\xb)
    \partial_{\Gamma}\left(u^\delta(\tau,\xb)-u^0(\tau,\xb) - \delta
    u^1\left(\tau,\xb,\frac{\xb}{\delta}\right)\right)^2 \, \mathrm{d}s(\xb) \mathrm{d}\tau\ .
  \end{aligned}
\end{equation}

  Applying the Newton-Leibniz theorem in the first term of
  $\Lambda^\delta$ leads to
  \begin{equation}\label{eq:Lambda_delta1}
    \begin{aligned}
      \Lambda^\delta (t) =& \delta \int_0^t \int_{\Omegad} \left(
        \partial_t{u}^\delta(\tau,\xb) - \partial_t{u}^0(\tau,\xb)\right)
      \left( u^\delta(\tau,\xb) - u^0(\tau,\xb)\right) \,
      \mathrm{d}s(\xb)\mathrm{d}\tau\\
      &+ \delta \int_0^t \int_{\Omegad} a(\xb)
      \partial_{\Gamma}(u^\delta(\tau,\xb)-u^0(\tau,\xb) - \delta
      u^1\left(\tau,\xb,\frac{\xb}{\delta}\right))^2 \, \mathrm{d}s(\xb) \mathrm{d}\tau.
    \end{aligned}
  \end{equation}
  Now we use the equation~(\ref{deltaP}) for the particular test
  function $v(\xb) := u^\delta(t,\xb) - u^0(t,\xb) \in
  \Hone(\Omegad)$ to resolve the quadratic terms in $u^\delta$, so
  that $\Lambda^\delta(t)$ becomes
  \begin{equation*}
    \begin{aligned}
      \Lambda^\delta (t) =&  -\delta \int_0^t \int_{\Omegad}
      \partial_t{u}^0(\tau,\xb) \left( u^\delta(\tau,\xb) -
        u^0(\tau,\xb)\right) \, \mathrm{d}s(\xb)\\
      &+ \delta \int_0^t \int_{\Omegad} a(\xb) \partial_{\Gamma} (-
      u^0(\tau,\xb) - \delta u^1\left(\tau,\xb,\frac{\xb}{\delta}\right))
      \partial_{\Gamma}(u^\delta(\tau,\xb)-u^0(\tau,\xb) - \delta
      u^1\left(\tau,\xb,\frac{\xb}{\delta}\right)) \, \mathrm{d}s(\xb) \mathrm{d}\tau\\
      &+ \delta \int_0^t \int_{\Omegad} a(\xb)
      \partial_{\Gamma}u^\delta(\tau,\xb) \partial_{\Gamma}( - \delta
      u^1\left(\tau,\xb,\frac{\xb}{\delta}\right)) \, \mathrm{d}s(\xb) \mathrm{d}\tau\\
      &- \delta \int_0^t \int_{\Omegad}
      f\left(\tau,\xb,\frac{\xb}{\delta}\right)
      (u^\delta(\tau,\xb)-u^0(\tau,\xb)) \, \mathrm{d}s(\xb) \mathrm{d}\tau .
    \end{aligned}
  \end{equation*}
  We use then the weak convergence stated by (\ref{conv1}) and the fact
  that $f$ is admissible.  Therefore $\Lambda^\delta(t)$ converges to
  the limit functional $\Lambda(t)$ defined by
  \begin{equation*}
    \begin{aligned}
      \Lambda(t) & := -\int_0^t \int_{\Omega} \int_{\Gamma_Y} a(\xb) (\nabla u^0(\tau,\xb) \cdot \tb(\yb) + \partial_{\Gamma,y} u^1(\tau,\xb,\yb))
      \partial_{\Gamma,y} u^1(\tau,\xb,\yb) \, \mathrm{d}\yb \,
      \mathrm{d}\xb \, \mathrm{d}\tau \\
      & = - \int_0^t \int_{\Omega} \bigg( \int_{\Gamma_Y}
      \partial_{\Gamma,\yb} \phib(\yb) \big( \tb(\yb) +
      \partial_{\Gamma,\yb} \phib(\yb) \big)^T \,
      \mathrm{d}s(\yb) \bigg) \nabla u^0(\tau,\xb) \cdot \nabla
      u^0(\tau,\xb) \,\mathrm{d}\xb \,\mathrm{d}\tau,\\
      & = - |\Gamma_Y|\int_0^t \int_{\Omega} \Abbh \nabla u^0(\tau,\xb) \cdot \nabla
      u^0(\tau,\xb) \,\mathrm{d}\xb \,\mathrm{d}\tau.
    \end{aligned}
  \end{equation*}
  The homogenized tensor $\Abbh$ is positive definite by Lemma~\ref{lA} and thus $\Lambda(t)\leq 0$.
  However, $\Lambda^\delta(t)\geq 0$ which is both possible only if  $\Lambda(t)=0$.
  Taking $t=\tf$ leads to
  \begin{equation*}
    \sqrt{\delta}\|\partial_{\Gamma} (u^\delta -
    u^0-\delta u^1)\|_{\Ltwo(0,\tf;\Ltwo(\Omegad))} \to 0.
  \end{equation*}
  We also deduce pointwise convergence
  \begin{equation*}
    \sqrt{\delta} \|u^\delta(t,\cdot)-u^0(t,\cdot) \|_{\Ltwo(\Omegad)}
    \to 0,
  \end{equation*}
  for any $t \in (0,\tf)$. Let us denote by $\Lambda^1_\delta (t)$
  and $\Lambda^\delta_2(t)$ the first and the second term,
  respectively, in the right hand side of
  \eqref{eq:Lambda_delta1}. Thus
  $\Lambda^\delta(t) = \Lambda^\delta_1(t) +
  \Lambda^\delta_2(t)$. For the second term we have
  $\sup_{t\in[0,\tf]} \Lambda_{2}^\delta = \Lambda_{2}^\delta(\tf)$ since the
  function under integral sign is positive. Thus pointwise
  convergence of $\Lambda^\delta$ implies convergence in $L^\infty$. For
  $\Lambda^\delta_1$ we show that it is equicontinuous.
  \begin{equation*}
    \begin{aligned}
      &\left|\Lambda^\delta_1(t+\Delta t) -
        \Lambda^\delta_1(t)\right| \leq  \delta \int_t^{t+\Delta t}
      \int_\Omegad \left|\partial_\tau
        (u^\delta(\tau,\xb)-u^0(\tau,\xb))
        (u^\delta(\tau,\xb)-u^0(\tau,\xb))\right|  \,
      \mathrm{d}s(\xb) \, \mathrm{d}\tau\\
      &\ \leq \left(\delta \int_t^{t+\Delta t} \int_\Omegad
        \left(\partial_\tau
          (u^\delta(\tau,\xb)-u^0(\tau,\xb))\right)^2 \,
        \mathrm{d}s(\xb) \, \mathrm{d}\tau\right)^{\frac{1}{2}} \\
        & \qquad \left(
        \delta \int_t^{t+\Delta t} \int_\Omegad
        (u^\delta(\tau,\xb)-u^0(\tau,\xb))^2 \,
        \mathrm{d}s(\xb) \, \mathrm{d}\tau\right)^{\frac{1}{2}}.
    \end{aligned}
  \end{equation*}
  From the \textit{a priori} estimates in Lemma~\ref{lapriori} we
  have that the first term is uniformly bounded with respect to
  $\delta$. For the second term we use the Newton-Leibnitz formula
  and obtain
  \begin{equation*}
    \begin{aligned}
      &\left|\Lambda^\delta_1(t+\Delta t) -
        \Lambda^\delta_1(t)\right| \leq C \left( \delta
        \int_t^{t+\Delta t} \int_0^\tau \partial_s  \int_\Omegad
        (u^\delta(s,\xb)-u^0(s,\xb))^2  \, \mathrm{d}s(\xb)\,
        ds \, \mathrm{d}\tau\right)^{\frac{1}{2}}\\
      &\ \leq C \left( 2 \delta
      \int_t^{t+\Delta t} \int_0^\tau
        \int_\Omegad \partial_{\zeta} (u^\delta(\zeta,\xb)-u^0(\zeta,\xb))
        (u^\delta(s,\xb)-u^0(\zeta,\xb))  \, \mathrm{d}s(\xb)\, d\zeta        \, \mathrm{d}\tau\right)^{\frac{1}{2}}\\
      &\ \leq C \left( 2 \int_t^{t+\Delta t} \Lambda^\delta_1(\tau)
        \, \mathrm{d}\tau\right)^{\frac{1}{2}}.
    \end{aligned}
  \end{equation*}
  Since $\Lambda^\delta_1$ is uniformly bounded we have that
  \begin{equation*}
    \left|\Lambda^\delta_1(t+\Delta t) - \Lambda^\delta_1(t)\right| \leq C (\Delta t)^{\frac{1}{2}}.
  \end{equation*}
  Thus $\Lambda^\delta_1$ is equicontinuous and since it pointwisely converges to 0 we get
  \begin{equation*}
    \| \Lambda^\delta\|_{L^\infty(0,\tf)} \to 0,
  \end{equation*}
  as $\delta$ tends to $0$. Thus~(\ref{eq:convergence_norms}) is
  proved.
\end{proof}


\section{Computation of the homogenized tensor $\Abbh$}
\label{sec5}
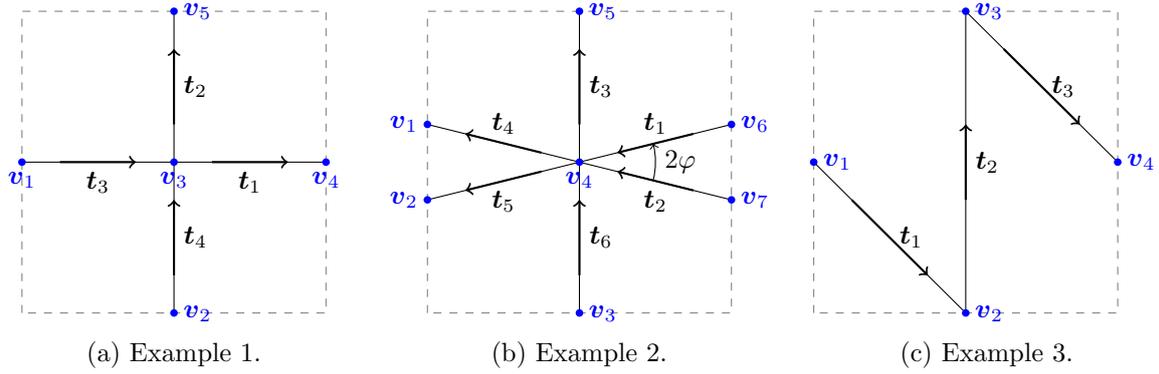
\begin{figure}[!h]

  \begin{tabular}{ccc}
  \begin{tikzpicture}
    \draw[dashed, black!50!white] (-2,-2) rectangle (2,2);
    \draw (-2,0) -- (2,0);
    \draw (0,-2) -- (0,2);

    \draw[thick,->] (-1.5,0) -- (-0.5,0) node [pos=0.5,below] {$\tb_3$};
    \draw[thick,->] (0.5,0) -- (1.5,0) node [pos=0.5,below] {$\tb_1$};
    \draw[thick,->] (0,-1.5) -- (0,-0.5) node [pos=0.5,right] {$\tb_4$};
    \draw[thick,->] (0,0.5) -- (0,1.5) node [pos=0.5,right] {$\tb_2$};

    \fill[blue] (2,0) circle (0.05) node [below] {$\vb_4$};
    \fill[blue] (0,2) circle (0.05) node [right] {$\vb_5$};
    \fill[blue] (-2,0) circle (0.05) node [below] {$\vb_1$};
    \fill[blue] (0,-2) circle (0.05) node [right] {$\vb_2$};
    \fill[blue] (0,0) circle (0.05) node [below] {$\vb_3$};

  \end{tikzpicture}
  &
  \begin{tikzpicture}
    \draw[dashed, black!50!white] (-2,-2) rectangle (2,2);

    \draw (0,-2) -- (0,2);
    \draw (-2,-0.5) -- (2,0.5);
    \draw (-2,0.5) -- (2,-0.5);

    \draw[thick,->] (1.5,0.375) -- (0.5,0.125) node [pos=0.5,above] {$\tb_1$};
    \draw[thick,->] (0,0.5) -- (0,1.5) node [pos=0.5,right] {$\tb_3$};
    \draw[thick,->] (-0.5,0.125) -- (-1.5,0.375) node [pos=0.5,above] {$\tb_4$};
    \draw[thick,->] (-0.5,-0.125) -- (-1.5,-0.375) node [pos=0.5,below] {$\tb_5$};
    \draw[thick,->] (0,-1.5) -- (0,-0.5) node [pos=0.5,right] {$\tb_6$};
    \draw[thick,->] (1.5,-0.375) -- (0.5,-0.125) node [pos=0.5,below] {$\tb_2$};

    \fill[blue] (-2,0.5) circle (0.05) node [left] {$\vb_1$};
    \fill[blue] (-2,-0.5) circle (0.05) node [left] {$\vb_2$};
    \fill[blue] (0,-2) circle (0.05) node [right] {$\vb_3$};
    \fill[blue] (0,0) circle (0.05) node [below] {$\vb_4$};
    \fill[blue] (0,2) circle (0.05) node [right] {$\vb_5$};
    \fill[blue] (2,0.5) circle (0.05) node [right] {$\vb_6$};
    \fill[blue] (2,-0.5) circle (0.05) node [right] {$\vb_7$};
    \draw[->] (-14:1) arc(-14:14:1) node[pos=0.5, right] {$2\varphi$};

  \end{tikzpicture}
  &
  \begin{tikzpicture}
    \draw[dashed, black!50!white] (-2,-2) rectangle (2,2);

    \draw (-2,0) -- (0,-2) -- (0,2) -- (2,0);


    \draw[thick,->] (-1.5,-0.5) -- (-0.5,-1.5) node [pos=0.5,right]
    {$\tb_1$};
    \draw[thick,->] (0,-0.5) -- (0,0.5) node [pos=0.5,right]
    {$\tb_2$};
    \draw[thick,->] (0.5,1.5) -- (1.5,0.5) node [pos=0.5,right]
    {$\tb_3$};

    \fill[blue] (-2,0) circle (0.05) node [right] {$\vb_1$};
    \fill[blue] (0,-2) circle (0.05) node [right] {$\vb_2$};
    \fill[blue] (0,2) circle (0.05) node [right] {$\vb_3$};
    \fill[blue] (2,0) circle (0.05) node [right] {$\vb_4$};

  \end{tikzpicture}
  \\
  (a) Example~\ref{example_plus}.
  &
  (b) Example~\ref{example_rhomb}.
  &
  (c) Example~\ref{example_blitz}.
  \end{tabular}

  \caption{The unit mesh pattern $\Gamma_Y$ of Examples~\ref{example_plus}--\ref{example_blitz}.}
  \label{fig:examples_1_2}
\end{figure}

In this section, we give a more practical way to compute the matrix
$\Abbh$ than solving~\eqref{canon}. Even though the geometries in Figure~\ref{fig:examples_1_2} are only with straight edges the following analysis also refers to curved edges.

\begin{lemma}\label{lexamples} The function $\phib=(\phi_1,\phi_2)$ is a solution
  to the canonical problem \eqref{canon} if and only if
  $s \mapsto \phib \circ \gamb_i(s) + \gamb_i(s)$ is affine, for all
  $i = 1 , \ldots, N_\cE$ and the Kirchhoff junction condition holds
  at all vertices, \ie,
  \begin{equation}\label{Kirchhoff}
    \sum_{\substack{i=1 \\ \gamb_i(\ell_i) = \vb_j}}^{N_\cE} (\phib
    \circ \gamb_i +\gamb_i) ' (\ell_i)  -
    \sum_{\substack{i=1 \\ \gamb_i(0) = \vb_j}}^{N_\cE} (\phib \circ
    \gamb_i +\gamb_i) ' (0)  = 0
  \end{equation}
  for all vertices $\vb_j \in \cV$, $j = 1, \ldots, N_\cV$.
\end{lemma}
\begin{proof}
We first prove that any solution of~\eqref{canon} is affine and then prove that it satisfies the Kirchhoff junction conditions.
If $\phib$ solves \eqref{canon}, then we have
  \begin{equation}
    \label{lexamples1}
    \begin{aligned}
      0 = & \int_{\Gamma_Y} (\partial_{\Gamma} \phib + \tb) \cdot \partial_{\Gamma}
      \psib \ \mathrm{d}s( \yb) = \sum_{i=1}^{N_\cE}
      \int_0^{\ell_i} \left[ \left( \partial_{\Gamma} \phib + \tb \right)
        \cdot
        \partial_{\Gamma} \psib  \right] \circ \gamb_i (s) \mathrm d s  \\
      = &\sum_{i=1}^{N_\cE} \int_0^{\ell_i} \left[ (\phib \circ
        \gamb_i ) ' (s) + \gamb_i' (s) \right] \cdot (\psib \circ
      \gamb_i)'(s)
      \mathrm{d} s.
    \end{aligned}
  \end{equation}
  In the last equation we see that if we test the expression with
  function   $\psib \in C_c^{\infty} (\gamb_i(\langle 0,\ell_i\rangle ),\ZR^2)$
  (for a particular $i$), after integration by parts, we obtain that the distributional derivative of
  $(\phib \circ \gamb_i ) ' (s) + \gamb_i' (s)$ is zero, so we get
  that $s \mapsto \phib \circ \gamb_i(s) + \gamb_i(s)$ is affine for
  all $i=1,\ldots,N_\cE$. Performing partial integration in
  (\ref{lexamples1}) now gives
  \begin{equation*}
    \sum_{i=1}^{N_\cE} \left[ (\phib \circ \gamb_i ) ' (s) +
      \gamb_i' (s) \right] \cdot (\psib \circ \gamb_i)(s)
    \Big|_0^{\ell_i} = 0.
  \end{equation*}
  Inserting continuous and affine test function $\psib$ having value
  $\eb_1$ (respectively $\eb_2$) at a vertex $\vb_j \in \cV$ and $0$
  at all other vertices we obtain the Kirchhoff junction conditions.

  To prove the converse follow exactly the opposite statements.
\end{proof}

In view of this lemma, the function
\begin{equation*}
  \xb \mapsto\qb(\xb):= \phib(\xb) + \xb
\end{equation*}
defined on $\Gamma_Y$ is affine on each edge, continuous on $\Gamma_Y$, but not periodic.

Before we proceed let us introduce some more definitions. For vertex
$\vb \in \cV$ we define $\etab(\vb)$ as the vector equal to $\eb_1$
or $\eb_2$ if $\vb$ lies on the right or top side of the unit cell (excluding the corners), respectively,
$\eb_1 + \eb_2$ if $\vb$ lies on the right top corner
and $\zerob$ otherwise.
Assuming that vertices in $\cV$ are indexed, with $\vb_{\pi(i)}$ we denote
vertex equal to $\vb_i$ if $\vb_i$ is not located on the upper or
right boundary of the unit cell, and otherwise the vertex $\vb_{i'}$
such that $\vb_i$ and $\vb_{i'}$ are opposite vertices. Thus we have
identities $\vb_{\pi(i)} = \vb_i -\etab(\vb_i)$ and
\begin{equation} \label{WeirdPeriodicity}
\qb(\vb_i) = \qb(\vb_{\pi(i)}) + \etab(\vb_i), \quad i \in \{1, \ldots, N_\cV \}.
\end{equation}

Since from this identity we see that values of function $\qb$ are uniquely defined in vertices on the right and upper boundary once they are defined in all other vertices, we define
\begin{equation}
  \label{eq:definition_b_Q}
  \b_j = \frac{\text{d}}{\text{d}s}(\qb\circ \gamb_j)(\cdot), \qquad \Qb_i = \qb(v_i),
\end{equation}
for $i=1,\ldots, N_{\cV'}$, $j=1,\ldots,N_\cE$, where $\cV'$ denotes the set
of all vertices not located on the right and upper boundary of the unit
cell. Without loss of generality, those vertices are indexed as first
$N_{\cV'}$ vertices in the set $\cV$. We will show that the matrix
$\Abbh$ can be found without solving the canonical problem \eqref{canon},
but by solving linear system of equations in terms of
$2N_\cE + 2 N_{\cV'}$ unknowns defined in \eqref{eq:definition_b_Q}.

We introduce incidence matrix
$\Abb_{\cI} \in M_{2N_{\cV'}, 2N_\cE} (\ZR)$ of the oriented periodic
graph $\Gamma_Y$, such that $2\times2$
submatrix on intersection of rows $2i-1,2i$ and columns $2j-1,2j$ is
equal to $\Ibb$ if the $j$-th edge enters the $i$-th vertex (\ie~$\gamb_j(\ell_j) = \vb_i$), equal to $-\Ibb$ if the $j$-th edge leaves
the $i$-th vertex (\ie~$\gamb_j(0) = \vb_i$), and $\0b$ otherwise. For a
similar argument see \cite{Zugec}. Thus, equations \eqref{Kirchhoff}
can be written, introducing the notation
$\b=(\b_1^T,\ldots,\b_{N_\cE}^T)^T$, by
\begin{equation}\label{examplesSystem1}
  \Abb_\cI \b = \0b.
\end{equation}
Secondly, for any edge $e_j$, $j = 1, \ldots, N_{\cE}$, which connects
vertices $ \vb_{i_1} = \gamb_j(\ell_j)$ and $\vb_{i_2} = \gamb_j(0)$
we use the definition of $\bb_j$ and $\Qb_i$ in~\eqref{eq:definition_b_Q} for a function $\qb$ and the Newton-Leibnitz
theorem  to obtain
\begin{equation}\label{Q}
  \begin{aligned}
    \Qb_{\pi(i_1)} - \Qb_{\pi(i_2)} &= \qb(\vb_{\pi(i_1)}) -
    \qb(\vb_{\pi(i_2)}) \\
    &= (\qb(\vb_{i_1}) - \qb(\vb_{i_2}))  - (\etab(\vb_{i_1}) - \etab(\vb_{i_2}))\\
    &= \int_{0}^{\ell_j} \frac{\mathrm{d}}{\mathrm{d}s}(\qb
    \circ \gamb_j)(s) \, \mathrm{d}s - (\etab(\gamb_j(\ell_j)) - \etab(\gamb_j(0))) \\
    & = \ell_j \b_j - (\etab(\gamb_j(\ell_j)) - \etab(\gamb_j(0))).
  \end{aligned}
\end{equation}
Hence, introducing the diagonal matrix
$\Lbb \in M_{2N_\cE, 2N_\cE} (\ZR)$ whose $(2j-1)$-th and $2j$-th
diagonal entries are equal to $\ell_j$, the vector $\fb \in \ZR^{2N_\cE}$
whose $(2j-1)$-th and $2j$-th component are
the first and the second component of
$\etab(\gamb_j(\ell_j)) - \etab(\gamb_j(0))$, respectively, and
$\Qb=(\Qb_1^T,\ldots,\Qb_{N_{\cV'}}^T)^T$ the equations~\eqref{Q} can be written as
\begin{equation} \label{examplesSystem2}
  \Abb_\cI^T \Qb + \Lbb \b = \fb.
\end{equation}
Equations \eqref{examplesSystem1} and \eqref{examplesSystem2} together form a linear system for $(\b,\Qb)$:
\begin{equation} \label{examplesSystem}
  \begin{bmatrix}
    \Lbb & \Abb_\cI^T \\
    \Abb_\cI & \0b
  \end{bmatrix}
  \begin{bmatrix}
    \b \\
    \Qb
  \end{bmatrix}
  =
  \begin{bmatrix}
    \fb \\
    \0b
  \end{bmatrix}.
\end{equation}
\begin{lemma} \label{lexamples1}
  System \eqref{examplesSystem} uniquely defines $\b$, and
  consequently it uniquely defines $\Abbh$.
\end{lemma}
\begin{proof}
  First,
  \eqref{examplesSystem2} is equivalent to
  \begin{equation} \label{examplesLema2eq1}
    \b = \Lbb^{-1} (\fb - \Abb_\cI^T \Qb).
  \end{equation}
  Inserting this expression in the second equation in \eqref{examplesSystem}  -- which is \eqref{examplesSystem1} -- we get
  \begin{equation}
    \label{examplesLema2eq2}
    \Abb_\cI \Lbb^{-1} \Abb_\cI^T \Qb = \Abb_\cI \Lbb^{-1} \fb.
  \end{equation}
  To show the solvability of this system, \ie~the existence of
  $\Qb$, we need to prove the condition from the Kronecker-Capelli
  theorem~\cite[p.~56]{shafarevich2012linear}:
  \begin{equation*}
    \Abb_\cI \Lbb^{-1} \fb \in \Imm{(\Abb_\cI \Lbb^{-1} \Abb_\cI^T)} =
    \Ker (\Abb_\cI \Lbb^{-1} \Abb_\cI^T)^\perp = \Ker
    (\Abb_\cI^T)^\perp = \Imm{\Abb_\cI},
  \end{equation*}
  which is clear. Thus, there is a solution to
  \eqref{examplesLema2eq2}. All solutions are described with
  $\Qb \in \Qb_0 + \Ker{\Abb_\cI^T}$ where $\Qb_0$ is a fixed
  solution.  Plugging it in back to \eqref{examplesLema2eq1} we get
  \begin{equation*}
    \b = \Lbb^{-1} (\fb - \Abb_\cI^T \Qb)  = \Lbb^{-1} (\fb - \Abb_\cI^T \Qb_0),
  \end{equation*}
  which is unique. Thus the matrix $\Abbh$ is uniquely defined as well since
  \begin{equation} \label{examplesLema2Mbb}
    \Abbh =  \frac{1}{|\Gamma_Y|} \int_{\Gamma_Y} (\partial_{\Gamma} \phib(\yb) +
    \tb(\yb))(\partial_{\Gamma} \phib(\yb) + \tb(\yb))^T \
    \mathrm{d} s(\yb)
    =  \frac{1}{|\Gamma_Y|} \sum_{j=1}^{N_\cE} \ell_j  \b_j \b_j^T,
  \end{equation}
  and $|\Gamma_Y| = \sum_{j=1}^{N_\cE} \ell_j.$
\end{proof}

As a contrary to $\b$, vector $\Qb$ is not uniquely defined by the
system \eqref{examplesSystem}, since due to Theorem~4.2.4 in
\cite{Z21}, $\Ker{\Abb_\cI^T}$ is two-dimensional. This is in
accordance with Lemma~\ref{lcanon}, since $\phib$ is defined
uniquely up to an additive constant.

\begin{remark}
  From the system~\eqref{examplesSystem}, from Lemma~\ref{lexamples1}
  and from the expression for $\Abbh$ we see that all properties of the
  homogenized problem come solely from the connectivity properties of the
  vertices in the oriented graph $\Gamma_Y$
  and lengthes of its edges. Thus, positions of vertices in the unit cell or
  differential geometry properties of edges forming the graph do not
  play a role in the definition of the matrix $\Abbh$ and thus in the homogenized problem.
\end{remark}

\begin{example}\label{example_plus}\em
  Let us find the operator $\Abbh$ for the geometry shown in
  Fig.~\ref{fig:examples_1_2}(a). The graph consists of $4$ edges and
  $5$ vertices with two of them being located on the right or upper
  boundary of the unit cell. That is the reason why the matrix
  $\Abb_\cI$ is $6\times8$ matrix and it is equal to
  \begin{equation*}
    \Abb_\cI = \begin{bmatrix}
      \Ibb_2 & \0b & -\Ibb_2 & \0b \\
      \0b & \Ibb_2 & \0b & -\Ibb_2 \\
      -\Ibb_2 & -\Ibb_2 & \Ibb_2 & \Ibb_2
    \end{bmatrix},
  \end{equation*}
  where  $\Ibb_n$ denotes the $n\times n$ identity
  matrix. Since all edges have equal length $\frac12$, the diagonal matrix $\Lbb$ is
  simply given by $\Lbb = \frac{1}{2} \Ibb_{2N_\cE}$, and, hence, $\Lbb^{-1} = 2
  \Ibb_{2N_\cE}$. As the first and second edge are ending on the right or upper boundary and, hence,
  $\etab(\gamb_1(\ell_1)) - \etab(\gamb_1(0)) = (1,0)^\top$ and
  $\etab(\gamb_2(\ell_2)) - \etab(\gamb_2(0)) = (0,1)^\top$ we find
  \begin{equation*}
    \fb = (1,0,0,1,0,0,0,0)^\top .
  \end{equation*}
  Now, we have everything explicitly defined to compute the homogenized tensor $\Abbh$.
  We proceed as in the proof of the Lemma
  \ref{lexamples1}. First, we seek one solution $\Qb^0$ of \eqref{examplesLema2eq2} which is in our case
  \begin{equation*}
    \begin{bmatrix}
      4 \Ibb_2 & \0b & -4\Ibb_2 \\
      \0b & 4\Ibb_2 & -4\Ibb_2 \\
      -4\Ibb_2 & -4\Ibb_2  & 8\Ibb_2
    \end{bmatrix}
    \Qb
    =
    2 \begin{bmatrix}
      \eb_1 \\
      \eb_2 \\
      - \eb_1 - \eb_2
    \end{bmatrix}.
  \end{equation*}
  Such a solution  is $\Qb^0 = \left(\frac{1}{2}, 0, 0,\frac{1}{2}, 0,0 \right)^\top$.
  Now, inserting $\Qb^0$ into \eqref{examplesLema2eq1} we get
    $\bb_1 = \bb_3 = \eb_1, \ \bb_2  = \bb_4 = \eb_2$,
  and finally inserting it into~\eqref{examplesLema2Mbb} leads to $\Abbh = \frac{1}{2} \Ibb_2$.
\end{example}

So, in the most simple unit-cell pattern, that was discussed in Example~\ref{example_plus}
the homogenized tensor is just a multiple of the identity matrix and in the homogenized
model we have just the two-dimensional Laplacian scaled by a factor.


In the following two examples we consider a unit-cell pattern
with high level of symmetry
for which the homogenized tensor $\Abbh$ is
a diagonal matrix but not just a multiple of the identity matrix $\Ibb_2$
or a non-diagonal matrix, respectively.

\begin{example}\em \label{example_rhomb}
  We consider the unit-cell graph $\Gamma_Y$ in
  Fig.~\ref{fig:examples_1_2}(b) that is composed of one vertical line and two lines
  that go through the center of the unit cell
  where the latter two intersect in an
  angle of $2\varphi$, where $\varphi \in (0,\frac{\pi}{4})$.
  The corresponding graph is composed of two edges of length $\frac 1 2$ and four edges
  of length $1/{(2 \cos \varphi)}$.

  Its matrices $\Abb_\cI$ and $\Lbb$ and vector $\fb$ are
\begin{equation*}
    \Abb_\cI = \begin{bmatrix}
      -\Ibb_2 & \0b & \0b & \Ibb_2 & \0b  & \0b \\
      \0b & -\Ibb_2 & \0b & \0b & \Ibb_2 & \0b \\
      \0b & \0b & -\Ibb_2 & \0b & \0b & \Ibb_2 \\
      \Ibb_2 & \Ibb_2 & \Ibb_2 & -\Ibb_2 & -\Ibb_2 & -\Ibb_2
    \end{bmatrix},
\end{equation*}
$$\Lbb = \diag{} \left(\frac{1}{2\cos \varphi}, \frac{1}{2\cos \varphi}, \frac{1}{2}, \frac{1}{2\cos \varphi}, \frac{1}{2\cos \varphi}, \frac{1}{2}\right) \otimes \Ibb_2,$$
and $\fb = (-1,0,-1,0,0,1,0,0,0,0,0,0)^T$.
  Again by solving first \eqref{examplesLema2eq1} and then plugging
  it in \eqref{examplesLema2eq1} and \eqref{examplesLema2Mbb} we get
  one particular solution $\Qb^0$, the vector $\bb$ and
  the homogenized tensor $\Abbh$:
  \begin{equation*}
    \Qb^0 = \frac{1}{2}\begin{bmatrix}
      \eb_1 \\
      \eb_1 \\
      -\eb_2 \\
      \0b
    \end{bmatrix}, \
    \bb = \begin{bmatrix}
      - \cos\varphi \eb_1 \\
      - \cos\varphi \eb_1 \\
      \eb_2 \\
      \cos\varphi \eb_1 \\
      \cos\varphi \eb_1 \\
      - \eb_2 \\
    \end{bmatrix}, \
    \Abbh = \frac{\cos \varphi}{\cos \varphi+2}\begin{bmatrix}
      2\cos \varphi & 0 \\
      0 & 1
    \end{bmatrix}.
  \end{equation*}
  For all angles $\varphi \in (0,\pi/4)$ the homogenized tensor~$\Abbh$
  is not a multiple of the identity matrix.
\end{example}

\begin{example}\em \label{example_blitz}
  We consider now the unit cell graph in Fig.~\ref{fig:examples_1_2}(c) for which the homogenized tensor $\Abbh$ is
  even not diagonal. It consists of $3$ edges and $4$ vertices with $2$ of them being on the right or upper
  side of unit cell. From that (and other properties of the graph) we
  see that
  \begin{equation*}
    \Abb_\cI = \begin{bmatrix}
      - \Ibb_2 & \0b & \Ibb_2 \\
      \Ibb_2 & \0b & - \Ibb_2
    \end{bmatrix}\ .
  \end{equation*}
  We notice that second edge of the graph is actually a loop in the
  corresponding periodic graph and consequently it does not appear in
  the matrix $\Abb_\cI$ since two identity matrices cancel out in matrix,
  on for entering and and one for leaving the same vertex.

  Moreover, the matrix $\Lbb$ and right-hand side $\fb$ are equal to
  \begin{align*}
    \Lbb &= \begin{bmatrix}
      \frac{\sqrt{2}}{2} \Ibb_2 & \0b & \0b \\
      \0b &  \Ibb_2 & \0b \\
      \0b & \0b & \frac{\sqrt{2}}{2} \Ibb_2
    \end{bmatrix}, &
    \fb &= \begin{bmatrix}
      \0b \\
      \eb_2 \\
      \eb_1-\eb_2
    \end{bmatrix}.
  \end{align*}

  In the same way as in the previous example, we find one particular solution $\Qb^0$, the vector $\bb$ and the homogenized tensor $\Abbh$:
  \begin{equation*}
    \Qb^0 = \begin{bmatrix}
      \frac{1}{4}(\eb_1 - \eb_2) \\
      - \frac{1}{4}(\eb_1 - \eb_2)
    \end{bmatrix}, \
    \bb = \begin{bmatrix}
      \frac{\sqrt{2}}{2}(\eb_1 - \eb_2) \\
      \eb_2 \\
      \frac{\sqrt{2}}{2}(\eb_1 - \eb_2)
    \end{bmatrix}, \
    \Abbh = \frac{1}{1+\sqrt{2}}\begin{bmatrix}
      \frac{\sqrt{2}}{2} & -\frac{\sqrt{2}}{2} \\
      -\frac{\sqrt{2}}{2} & \frac{\sqrt{2}}{2}+1
    \end{bmatrix},
  \end{equation*}
  where the latter is not diagonal.
\end{example}

\begin{figure}[!bt]
  \begin{tabular}{ccc}
    \includegraphics[height=0.3\linewidth]{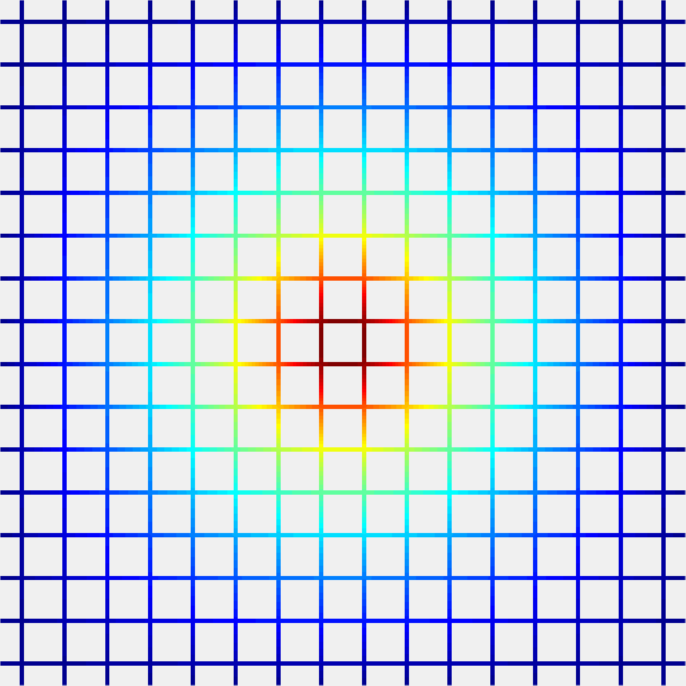}
    &
    \includegraphics[height=0.3\linewidth]{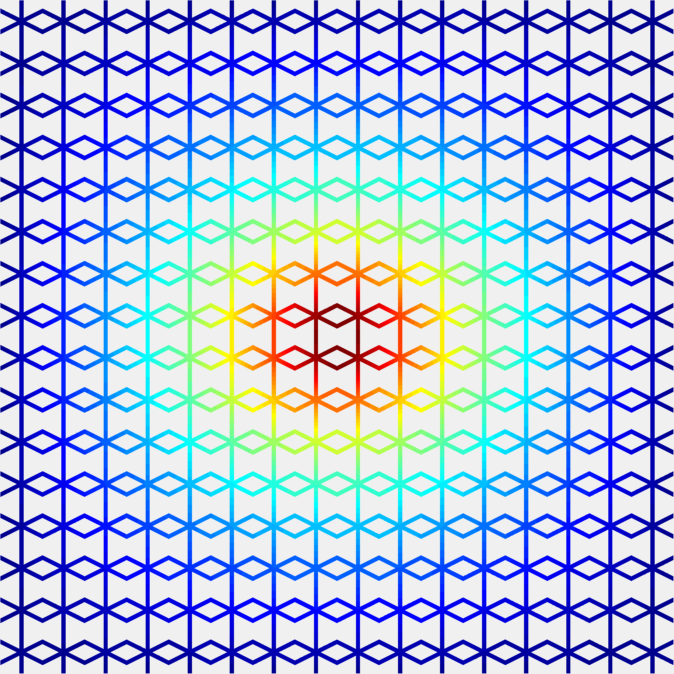}
    &
    \includegraphics[height=0.3\linewidth]{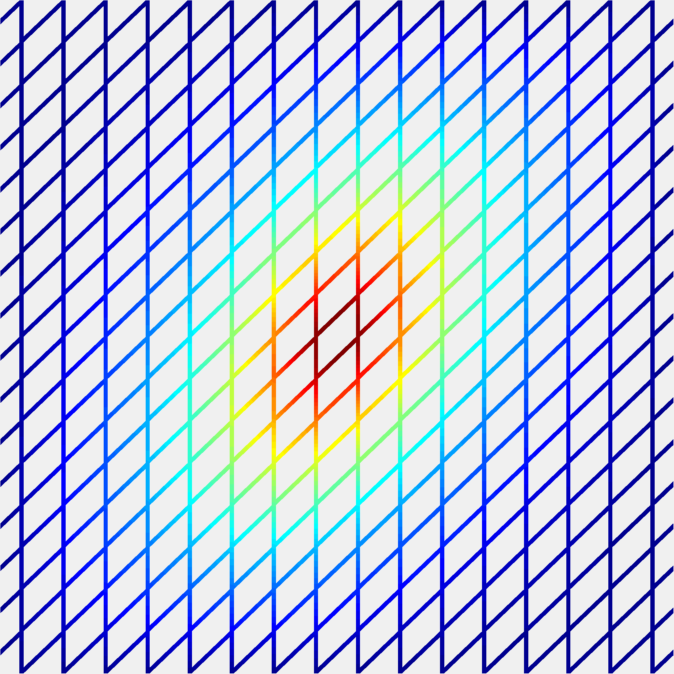}\\
    \includegraphics[height=0.3\linewidth]{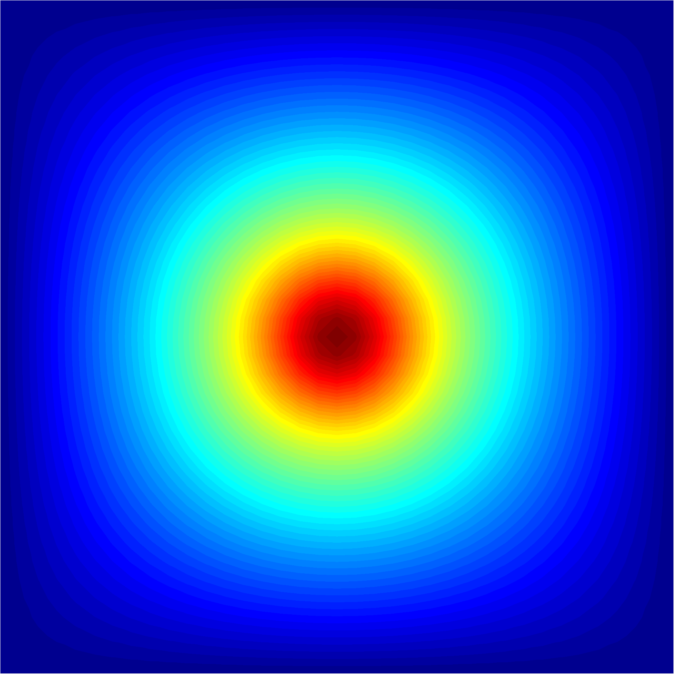}
    &
    \includegraphics[height=0.3\linewidth]{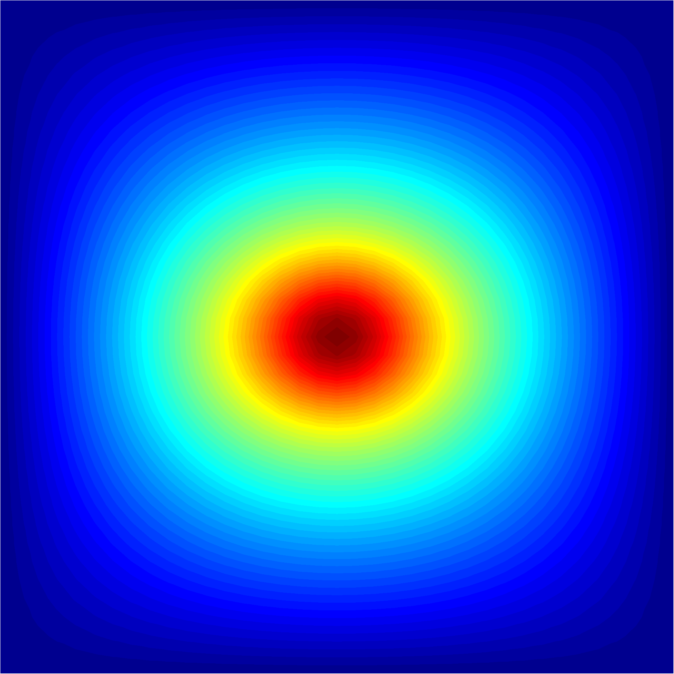}
    &
    \includegraphics[height=0.3\linewidth]{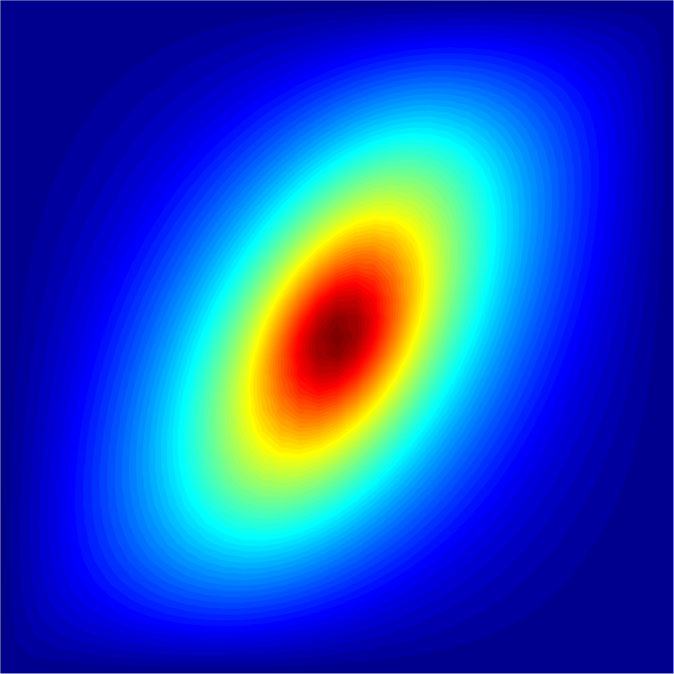}\\[0.5em]
    (a) Example~\ref{example_plus}. &
    (b) Example~\ref{example_rhomb}. &
    (c) Example~\ref{example_blitz}.
  \end{tabular}
  \caption{Temperature distribution on the mesh~$\Omegad$ ({\em top row}) and for the homogenized model on the domain~$\Omega = (0,1)^2$ ({\em bottom row}) for $\delta = 1/16$ at time $\tf = 2$.
    The same colorbar is used in all figures that scales from $0$ (blue) to $1.4\cdot 10^{-4}$ (red).}
  \label{fig:plots_case1}
\end{figure}

\section{Numerical experiments}
\label{sec:numerics}

In this section we compare the solutions of the derived homogenized model~\eqref{P} and corresponding $\delta$-problem (\ref{deltaP}) for the three pattern
introduced in Example~\ref{example_plus}--\ref{example_blitz} by numerical
simulations using the numerical C++ library
Concepts~\cite{Frauenfelder.Lage:2002}.

For this we apply first a semi-discretization in space using the finite element method
with continuous and piecewise polynomial functions for both,
the $\delta$-dependent problem~\eqref{deltaP} on the mesh $\Omegad$
as well as for the homogenized problem~\eqref{P} on the domain $\Omega$.
For this each edge of the the mesh $\Omegad$ is subdivided into smaller edges to obtain
a one-dimensional finite element mesh and the domain $\Omega$ is partitioned into non-overlapping quadrilateral cells.

Denoting the stiffness matrix of the semi-discretized $\delta$-dependent problem~\eqref{deltaP}
by $\Kbb^\delta$, the mass matrix by $\Mbb^\delta$, the
right hand side vector by $F^\delta(t)$,
the time-dependent solution vector (associated to $u^\delta(t,\cdot)$) by $U^\delta(t)$
and the initial vector (associated to $u_\init$) by $U^\delta_\init$
the semi-discrete problem reads
\begin{equation}
  \label{eq:heat_delta_half_discretized}
  \Mbb^\delta \partial_t{U}^\delta(t) + \Kbb^\delta U^\delta(t) =
  F^\delta(t), \quad t > 0, \qquad U^\delta(0) = U^\delta_\init.
\end{equation}
To impose the Dirichlet boundary conditions a penalization is used.

For the time-discretization of (\ref{eq:heat_delta_half_discretized}) we are
using the Crank-Nicholson scheme that is of second order in the (uniform) time
step $\Delta_T$.  Then, with $F^\delta_n = F^\delta(n \Delta_T)$ the solution
vectors $U^\delta_n = U^\delta(n \Delta_T)$ at time $t_n = n \Delta_T$ fullfil
the linear systems
\begin{equation}
  \label{eq:heat_delta_full_discretized}
  \Big( \Mbb^\delta + \tfrac{\Delta_T}{2}
  \Kbb^\delta \Big) U^\delta_{n+1} = \Big( \Mbb^\delta  - \tfrac{\Delta_T}{2}
  \Kbb^\delta \Big) U^\delta_n + \tfrac{\Delta_T}{2} \Mbb^\delta
  \big( F^\delta_{n+1} + F^\delta_n \big), \quad n \in \ZN\ ,
\end{equation}
with the initial condition $U^\delta_0 = U^\delta_\init$.

Discretizing the homogenized problem~\eqref{P} in the same way we obtain
solution vectors $U^0_n$ at time $t_n = n \Delta_T$.  We project this discrete
solution at each time onto the finite element mesh of $\Omegad$, where
$\cP^\delta(U^0_n)$ is the projection of the vector $U^0_n$, to compute an
approximation of the relative $\Ltwo(\Omegad)$ error
\begin{equation*}
  \cE(t_n) := \sqrt{\Big( \big( U^\delta_n - \cP^\delta(U^0_n)
    \big)^T \Mbb^\delta \big( U^\delta_n - \cP^\delta(U^0_n) \big) \Big) /
    \Big( \big( U^\delta_n \big)^T \Mbb^\delta U^\delta_n \Big) }.
\end{equation*}

For all simulations we consider the domain
$\Omega = (0,1)^2$, the coefficients $\rho c_p = a = 1$,
the source term
\begin{equation*}
  f^\delta(t,\xb) = f(t,\xb,\tfrac{\xb}{\delta}) = 4\exp(-196|\xb-(0.5,0.5)|^2)\exp(-3t)
\end{equation*}
that depends only on the macroscopic variable and, hence,
$f_{\homg} = f^\delta$, and the initial data $u_\init = 0$. We found the
discretization error small in comparison to the modelling error when splitting each edge of
$\Gammad(n_1,n_2)$ uniformly into three (smaller) edges on which polynomials of
degree $2$ are used. In this way the
discretization error decreases with decreasing period $\delta$.
For the homogenized solution, we use a uniform mesh of 16 square cells and
polynomials of degree $6$. For the time
discretization of the Crank-Nicholson scheme we used as time step
$\Delta_T=0.002$.

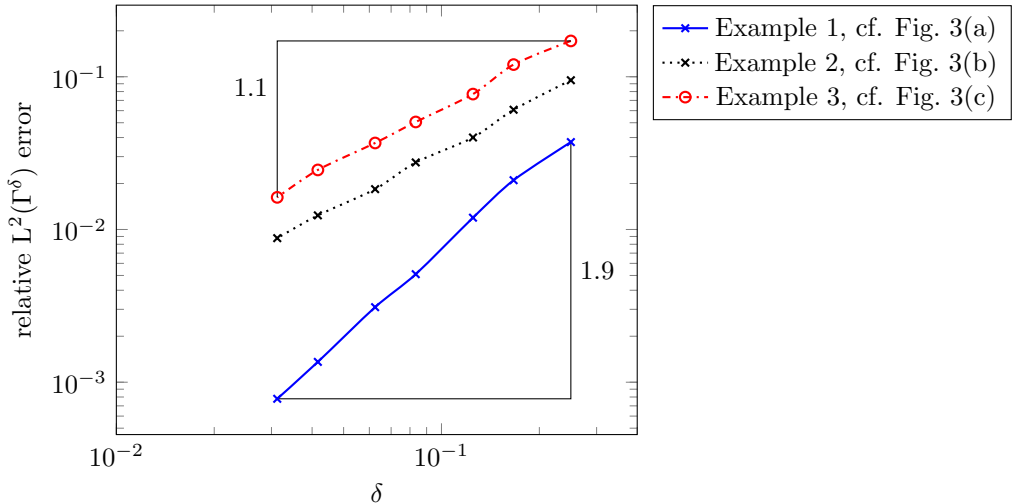
\begin{figure}[!bt]
  \centering
  \null\hfill
  \begin{tikzpicture}
    \begin{loglogaxis}[
      xlabel=$\delta$,
      ylabel=relative $\Ltwo(\Omegad)$ error,
      xmin=0.01,
      xmax=0.4,
      legend pos=outer north east]
      \addplot+[color=blue,thick,smooth,mark=x, mark options={solid, color=blue}] table [x = delta, y = error]{L2_delta_example1_decaying_cross.txt};
      \addlegendentry{Example~\ref{example_plus}, cf. Fig.~\ref{fig:examples_1_2}(a)};
      \addplot+[color=black,thick,dotted,smooth,mark=x, mark options={solid, color=black}] table [x = delta, y = error]{L2_delta_example1_decaying_one_bow.txt};
      \addlegendentry{Example~\ref{example_rhomb}, cf. Fig.~\ref{fig:examples_1_2}(b)};
      \addplot+[color=red,thick,dashdotted,smooth,mark=o, mark options={solid, color=red}] table [x = delta, y = error]{L2_delta_example1_decaying_blitz.txt};
      \addlegendentry{Example~\ref{example_blitz}, cf. Fig.~\ref{fig:examples_1_2}(c)};
      \draw ({axis cs:0.03125,0.000778905835267}) -| ({axis cs:0.25,0.037349702895225}) node[pos=0.75,right] {$1.9$};
      \draw ({axis cs:0.03125,0.016231756182536}) |- ({axis cs:0.25,0.171603212815735}) node[pos=0.35,left] {$1.1$};
    \end{loglogaxis}
  \end{tikzpicture}
  \hfill\null
  \caption{Plot of the relative $\Ltwo(\Omegad)$ error $\cE(\tf)$ with respect to the period
    $\delta$ for $\tf=2$ for the three unit cell patterns in Fig.~\ref{fig:examples_1_2}. For Example~\ref{example_rhomb} the angle is
    $\phi=\tan^{-1}(0.5)$.}
  \label{fig:L2_error_case1}
\end{figure}

In Fig.~\ref{fig:plots_case1} the discretized solution of the heat
equation~(\ref{eq:Heat_equation}) on the mesh $\Omegad$ with $\delta = 1/16$ --
for illustrative purposes on a thickened graph of thickness $1/160$ -- and the
homogenized solution,\ie the solution of~(\ref{P}), is illustrated for the
Examples~\ref{example_plus}--\ref{example_blitz} at time $\tf = 2$.  For
Example~\ref{example_rhomb} the angle is $\phi=\tan^{-1}(0.5)$.  Note, that in
all figures in this section the same colorbar is used that scales from $0$
(blue) to $1.4\cdot 10^{-4}$ (red). %
For Example~\ref{example_plus}, for which the homogenized tensor is a multiple
of the identity matrix $\Ibb_2$, we observe the temperature decays from the
mid-point approximately the same way in all directions.  For
Example~\ref{example_rhomb}, for which the homogenized tensor is diagonal but
not a multiple of $\Ibb_2$, we observe a faster decay in one axis direction. %
For Example~\ref{example_blitz}, for which the homogenized tensor is not even
diagonal, we observe a faster decay in another direction.  In all three
examples the homogenized solution is in very good agreement with the solution
of the heat equation on the mesh~$\Omegad$.

In Fig.~\ref{fig:L2_error_case1} the (approximative) relative
$\Ltwo(\Omegad)$-error $\cE(t_n)$ is shown as a function of the
period~$\delta$. We observe a linear convergence to $0$ for
Example~\ref{example_rhomb} and Example~\ref{example_blitz}.
Moreover for Example~\ref{example_plus}
the limit solution shows a quadratic convergence in $\delta$ that goes beyond
the theory.

\begin{figure}[!tb]
  \null
  \hfill
  \begin{tikzpicture}
    \draw[dashed, black!50!white] (-2,-2) rectangle (2,2);
    \draw[white] (-2,-4) -- (0,2);

    \draw (-2,-2) -- (2,2);
    \draw (-2,2) -- (2,-2);
    \draw[thick,<-] (0.5,0.5) -- (1.5,1.5) node [pos=0.5,below] {$\tb_2$};
    \draw[thick,->] (-0.5,0.5) -- (-1.5,1.5) node [pos=0.5,below] {$\tb_1$};
    \draw[thick,->] (-1.5,-1.5) -- (-0.5,-0.5) node [pos=0.5,below] {$\tb_4$};
    \draw[thick,<-] (1.5,-1.5) -- (0.5,-0.5) node [pos=0.5,below] {$\tb_3$};

    \fill[blue] (2,2) circle (0.05) node [left] {$\vb_4$};
    \fill[blue] (-2,-2) circle (0.05) node [right] {$\vb_1$};
    \fill[blue] (2,-2) circle (0.05) node [left] {$\vb_2$};
    \fill[blue] (0,0) circle (0.05) node [right] {$\vb_3$};
    \fill[blue] (-2,2) circle (0.05) node [right] {$\vb_5$};
  \end{tikzpicture}
  \hfill
  \begin{tikzpicture}
    \draw[dashed, black!50!white] (-2,-2) rectangle (2,2);
    \draw[white] (-2,-4) -- (0,2);

    \draw (-2,0) -- (0,-2) -- (2,0) -- (0,2) -- (-2,0);
    \draw[thick,->] (1.5,0.5) -- (0.5,1.5) node [pos=0.5,below] {$\tb_1$};
    \draw[thick,->] (-0.5,1.5) -- (-1.5,0.5) node [pos=0.5,below] {$\tb_2$};
    \draw[thick,->] (-1.5,-0.5) -- (-0.5,-1.5) node [pos=0.5,below] {$\tb_3$};
    \draw[thick,->] (0.5,-1.5) -- (1.5,-0.5) node [pos=0.5,below] {$\tb_4$};

    \fill[blue] (2,0) circle (0.05) node [left] {$\vb_4$};
    \fill[blue] (-2,0) circle (0.05) node [right] {$\vb_1$};
    \fill[blue] (0,-2) circle (0.05) node [right] {$\vb_2$};
    \fill[blue] (0,2) circle (0.05) node [right] {$\vb_3$};
  \end{tikzpicture}
  \hfill
  \null\\[-4em]
  \null
  \hfill
  \includegraphics[height=4cm]{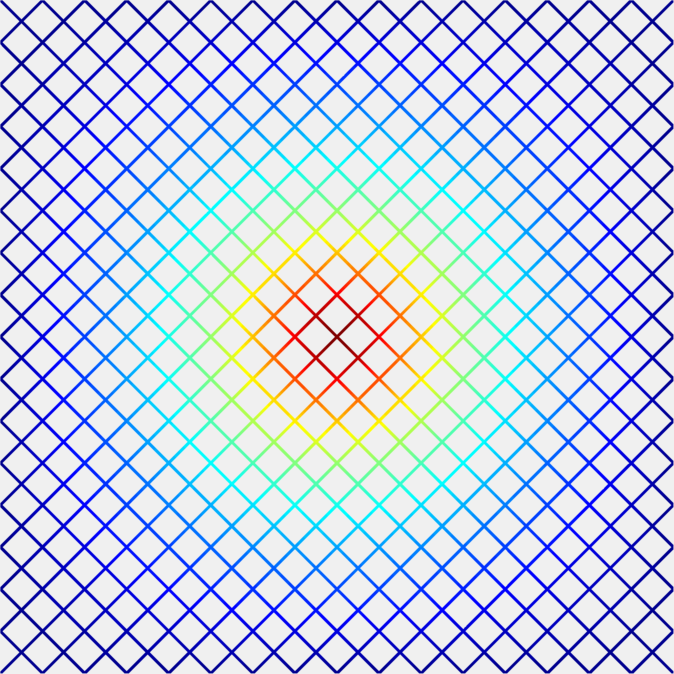}
  \hfill
  \includegraphics[height=4cm]{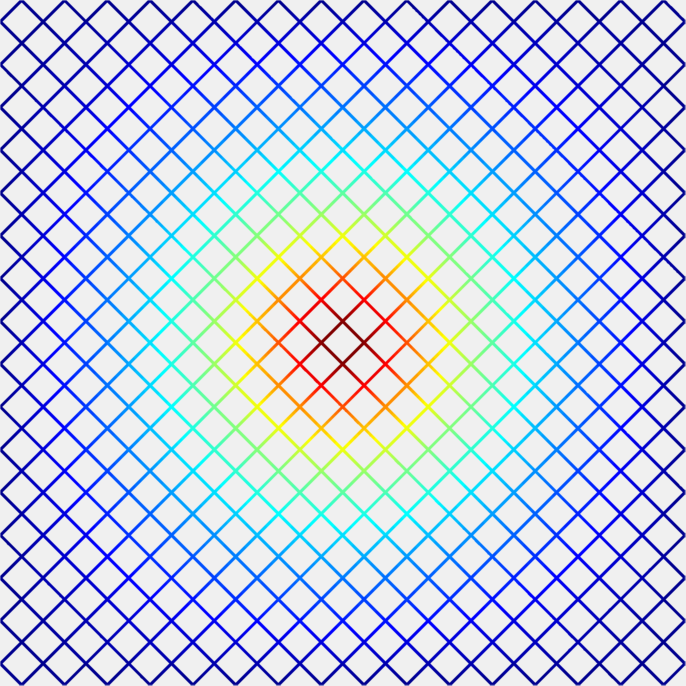}
  \hfill
  \null
  \caption{Two pattern $\Gamma_Y$ with the  same homogenized tensor $\Abbh = \tfrac{1}{2} \Ibb_2$ as the one of
    Example~\ref{example_plus} ({\em top row}) and temperature distribution on the mesh $\Omegad$
    for $\delta=1/16$ at time $\tf=2$ ({\em bottom row}).}
  \label{fig:other_patterns_plus}
\end{figure}
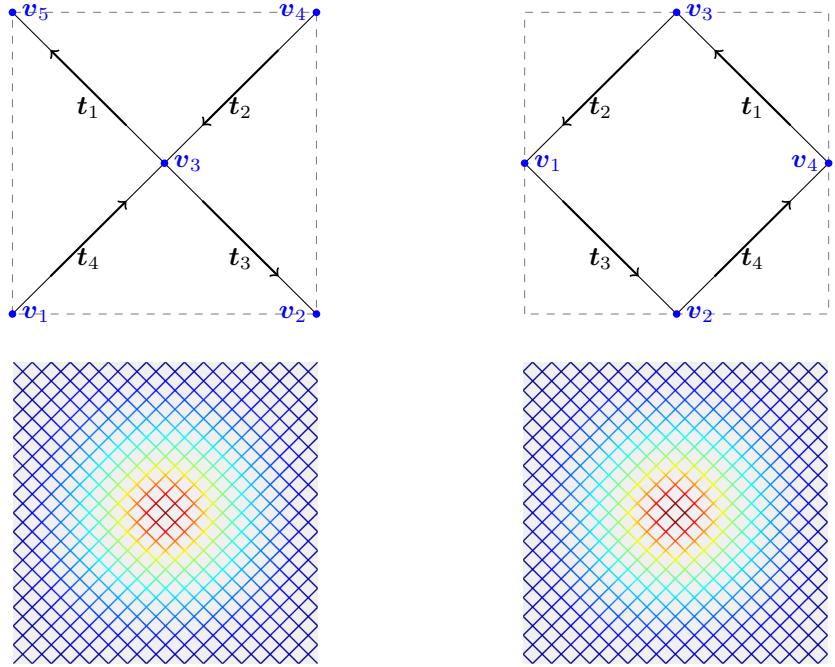

For the curiosity of the reader we show two patterns in
Fig.~\ref{fig:other_patterns_plus} that extended in all direction represent the
same mesh as the mesh of pattern in Example~\ref{example_plus},
cf. Fig.~\ref{fig:examples_1_2}(a). %
One can easily verify that for the two patterns the homogenized tensor equals
$\tfrac{1}{2} \Ibb_2$ as for the one of Example~\ref{example_plus}, \ie, the
solution on the mesh~$\Omegad$ has macroscopically at leading order the same
behaviour.  This we observe in the numerical experiments
(cf. Fig.~\ref{fig:plots_case1}(a) and Fig.~\ref{fig:other_patterns_plus}).


%
%
%

\section*{Acknowledgement}
The authors gratefully acknowledge the financial support of their
  research cooperation "Asymptotic and algebraic analysis of nonlinear eigenvalue problems for biomechanical and photonic devices"
  through the bilateral program "Procope"
  between the German Academic Exchange Service (DAAD) based on funding
  of the German Federal Ministry of Education and Research (project ID 57334847) and the Croatian Ministry of
  Science (contract number 910-08/16-01/00209).  The research was
  partly conducted during the stay of the first author at the Graduate School
  Computational Engineering (CE) at the Technical University (TU) Darmstadt and he is grateful to the support and
  hospitality.  The authors would like to thank Luka Grubi\v{s}i\'{c}
  (University of Zagreb) and Herbert Egger (TU Darmstadt) for fruitful
  discussions.

\end{document}